\documentclass[a4paper, titlepage, oneside, 11pt]{amsart}
\usepackage{amsmath,amssymb,amsfonts,amstext,amscd,array,cite,enumitem,geometry,graphicx,latexsym,mathptmx,xcolor,pstricks}
\usepackage[utf8]{inputenc}
\usepackage{newunicodechar}
\usepackage[all,cmtip]{xy}

\newtheorem{lem}{Lemma}[section]
\newtheorem{prop}[lem]{Proposition}
\newtheorem{thm}[lem]{Theorem}
\newtheorem{cor}[lem]{Corollary}

\newtheorem{rem}[lem]{Remark}

\setlist[description]{leftmargin=\parindent,labelindent=0cm}

\newcommand{\la}{\lambda}
\def\R{{\mathbb{R}}}
\def\N{{\mathbb{N}}}
\def\Z{{\mathbb{Z}}}
\def\a{{\alpha}}
\def\b{{\beta}}
\def\Gam{\mathrm{Gam}}
\def\tBeta{\mathrm{Be}'}
\def\z{\mathrm{zero}}
\def\AL{\mathrm{AL}}
\def\E{\mathrm{E}}

\newcommand{\vertiii}[1]{{\left\vert\kern-0.25ex\left\vert\kern-0.25ex\left\vert #1
    \right\vert\kern-0.25ex\right\vert\kern-0.25ex\right\vert}}

\numberwithin{equation}{section}

\allowdisplaybreaks

\begin{document}

\title[Stationary solutions of random polymer models]{On the stationary solutions\\of random polymer models\\and their zero-temperature limits}

\author[D.~A.~Croydon]{David A. Croydon}
\address{Research Institute for Mathematical Sciences, Kyoto University, Kyoto 606-8502, Japan}
\email{croydon@kurims.kyoto-u.ac.jp}

\author[M.~Sasada]{Makiko Sasada}
\address{Graduate School of Mathematical Sciences, University of Tokyo, 3-8-1, Komaba, Meguro-ku, Tokyo, 153--8914, Japan}
\email{sasada@ms.u-tokyo.ac.jp}

\begin{abstract}
We derive stationary measures for certain zero-temperature random polymer models, which we believe are new in the case of the zero-temperature limit of the beta random polymer (that has been called the river delta model). To do this, we apply techniques developed for understanding the stationary measures of the corresponding positive-temperature random polymer models (and some deterministic integrable systems). More specifically, the article starts with a survey of results for the four basic beta-gamma models (i.e.\ the inverse-gamma, gamma, inverse-beta and beta random polymers), highlighting how the maps underlying the systems in question can each be reduced to one of two basic bijections, and that through an `independence preservation' property, these bijections characterise the associated stationary measures. We then derive similar descriptions for the corresponding zero-temperature maps, whereby each is written in terms of one of two bijections. One issue with this picture is that, unlike in the positive-temperature case, the change of variables required is degenerate in general, and so whilst the argument yields stationary solutions, it does not provide a complete characterisation of them. On the other hand, this degeneracy does allow us to explain the appearance of atoms in the stationary measures of certain zero-temperature models. We also derive from our viewpoint various links between random polymer models, some of which recover known results, some of which are novel, and some of which lead to further questions.

\bigskip

\noindent
\emph{Data sharing not applicable to this article as no datasets were generated or analysed during the current study.}
\end{abstract}



\date{\today}

\maketitle

\setcounter{tocdepth}{2}
\tableofcontents

\section{Introduction}

In the recent work \cite{CS2}, a general approach for identifying the spatially independent and identically distributed stationary measures of various classical discrete integrable systems, namely the discrete and ultra-discrete KdV equations, and the discrete and ultra-discrete Toda lattice, was presented. Moreover, parallels were drawn between the argument of \cite{CS2} and the techniques of \cite{CN}, which characterised the stationary measures of a class of random polymer models. In particular, in both \cite{CS2} and \cite{CN} appeals were made to particular decompositions of the recursive maps that define the models of interest. After presenting an extended introduction to what is known in the positive-temperature random polymer context, the aim of this article is to highlight the use of similar decompositions in understanding zero-temperature limits of the models in question. Whilst the results we present were for the most part previously known (including all those relevant to the positive-temperature systems), the stationary solutions we derive for the zero-temperature limit of the beta random polymer are apparently new. We further believe that the systematic approach set out here helps clarify why particular stationary measures arise for the different random polymer models considered, and the connections between them. Indeed, in both the positive- and zero-temperature settings, each of the systems is related to one of two basic bijections (which are themselves related), and each of these is naturally connected to certain distributions via an `independence preservation' property. Moreover, the appearance in atoms in the stationary measures of various zero-temperature polymer models is explained in terms of a collapse of part of the domain from the positive- to zero-temperature regime.

Let us start by introducing the central ideas of \cite{CS2} in the context of two-dimensional stochastic lattice models. In the latter setting we commonly have that the input to a system is provided by an array $(X_{n,m})_{n,m \in \Z^2}$ of random variables, and the behavior of the partition function $(Z_{n,m})_{n,m \in \Z^2}$, as determined by $(X_{n,m})_{n,m \in \Z^2}$, is the main object of interest; see Section \ref{sec2} for the concrete random polymer examples that will be our primary focus. Writing $(U_{n,m})_{n,m \in \Z^2}$ and $(V_{n,m})_{n,m \in \Z^2}$ for the `derivative' (in a suitable sense) of $(Z_{n,m})_{n,m \in \Z^2}$ in the first- and second-coordinate directions, respectively, one then typically has a recursive relation of the form:
\begin{equation}\label{R-eq}
R\left(X_{n,m}, U_{n,m-1},V_{n-1,m}\right)=\left(U_{n,m},V_{n,m}\right),
\end{equation}
where $R: J_1 \times I_1 \times I_2 \to I_1 \times I_2$ for some subsets $J_1, I_1,I_2\subseteq \mathbb{R}$ (usually intervals or discrete subsets thereof). In the `upper-right corner'-version of such a model, one takes independent random variables $(X_{n,m})_{n,m \in \N}$, $(U_{n,0})_{ n \in \N}$ and $(V_{0,n})_{n \in \N}$, and from these inputs determines the remaining variables $(U_{n,m})_{n, m \in \N}$ and $(V_{n,m})_{n, m \in \N}$ by applying \eqref{R-eq} repeatedly. Given this structure, if there exists a triplet of probability measures $(\tilde{\mu},\mu, \nu)$ on $J_1$, $I_1$ and $I_2$ such that
\begin{equation}\label{Rinv}
R\left( \tilde{\mu} \times \mu \times \nu\right)=\mu \times \nu,
\end{equation}
where an expression of the form $\mu \times \nu$ means a product of measures, then the stochastic lattice model has a stationary measure;
see the discussion in \cite{CS2}, for example. (NB.\ For a measurable map $f$ and measure $\pi$, we write $f(\pi):=\pi\circ f^{-1}$ for the push-forward of $\pi$ by $f$.) Hence it is important to find such a collection of measures.

It transpires that for the random polymer models studied in \cite{CN} the map $R$ can be usefully decomposed into two maps, each with two inputs and two outputs. More precisely, one can find a bijection $F:I_1\times I_2\rightarrow J_1\times J_2$, where $J_2\subseteq\mathbb{R}$, such that $R$ can be reconstructed from $F$ and $F^{-1}$. To this end, one first defines an involution $\overline{F}: J_1 \times I_1 \times I_2 \to J_1 \times I_1 \times I_2$ by setting
\begin{equation}\label{Fbar}
\overline{F}(a,b,c):=\left( F^{(1)}(b,c), F^{-1} (a, F^{(2)}(b,c) \right),
\end{equation}
where we write $F(b,c)=(F^{(1)}(b,c),F^{(2)}(b,c))$. That is, $(d,e,f):=\overline{F}(a,b,c)$ can be computed by first applying $F$ to $(b,c)$ to give $(d,g)=F(b,c)$, and then applying $F^{-1}$ to $(a,g)$ to give $(e,f)=F^{-1}(a,g)$. If $F$ is chosen appropriately, one then finds that $R=\overline{F}^{(2,3)}$, i.e.\
\begin{equation}\label{FR}
R(a,b,c)=\left(\overline{F}^{(2)}(a,b,c),\overline{F}^{(3)}(a,b,c)\right).
\end{equation}
Pictorially, keeping track of the lattice variables, this situation can be represented as follows:
\begin{equation}\label{picture}
\hspace{40pt}\xymatrix@C-15pt@R-15pt{\boxed{F}\vphantom{F^{-1}} & d &\boxed{F^{-1}}& e=U_{n,m}&\\
c=V_{n-1,m} \ar[rr] && g\ar[rr]&& f=V_{n,m},\\
& b=U_{n,m-1} \ar[uu]&&a=X_{n,m}\ar[uu]&}
\end{equation}
We will recall below explicitly the decomposition of $R$ for the random polymer models of \cite{CN}. Moreover, we note that in \cite[Section 6]{CS2}, it was observed that if $R$ is the map arising for a certain directed polymer with site weights, then the map $\overline{F}$ is that associated with the discrete Toda lattice, and, similarly, the $R$ of directed last passage percolation corresponds to the $\overline{F}$ associated with the ultra-discrete Toda lattice. (Precisely, up to a simple change of variables, discrete Toda corresponds to the $\overline{F}$ given by $R_{(0,1)}$, and ultra-discrete Toda corresponds to the $\overline{F}$ given by $R^\z_{(0,1)}$, using the notation introduced subsequently at \eqref{rab} and \eqref{ztl}.) As will become clear in what follows, these connections explain why the systems admit related stationary measures.

The advantage of the above picture is that it gives a route to understanding the stationary measures of $R$, i.e.\ those satisfying \eqref{Rinv}, via the solutions of a certain `detailed balance condition' for $F$. In particular, we say that a quadruplet $(\mu, \nu,\tilde{\mu},\tilde{\nu})$ on $I_1$, $I_2$, $J_1$ and $J_2$ satisfies the \emph{detailed balance condition} for $F$ if
\begin{equation}\label{db}
F\left(\mu\times \nu\right)=\tilde{\mu}\times\tilde{\nu}.
\end{equation}
This condition was called the `detailed balance condition for a type II model' in \cite{CS2}, wherein it was shown to be equivalent to the temporal invariance in distribution of the corresponding bi-infinite dynamical system started from $(\mu\times\tilde{\mu})^\mathbb{Z}$. Whilst this usage of the terminology detailed balance is different to the standard one for reversible Markov chains, it similarly provides, in the setting of \cite{CS2}, a local condition for the global invariance of a system. Notice also that the detailed balance condition of \eqref{db} can be seen as the aforementioned independence preservation property for the bijection $F$. A key result in \cite{CS2} is the following proposition, which shows that solutions of the detailed balance condition for $F$ precisely correspond to triplets of measures $(\tilde{\mu},\mu, \nu)$ whose product is invariant under $\overline{F}$. Subsequently, applying the link between $R$ and $\overline{F}$ from \eqref{FR}, one sees that invariant product measures for $\overline{F}$ yield solutions of \eqref{Rinv}, and vice versa. Thus we have the desired connection between the detailed balance condition for $F$ and stationary measures for $R$.

\begin{prop}[{\cite[Proposition 2.8]{CS2}}]\label{twothree}
Let $F:I_1\times I_2\rightarrow J_1\times J_2$ be a bijection, where $I_1,I_2,J_1,J_2$ are measurable subsets of $\mathbb{R}$, and define $\overline{F}$ as at \eqref{Fbar}. For a triplet of probability measures $(\tilde{\mu}, \mu, \nu)$ on $J_1$, $I_1$ and $I_2$, the following three conditions are equivalent:
\begin{enumerate}
\item[(a)] $\overline{F}^{(2,3)}( \tilde{\mu} \times \mu \times \nu)=\mu \times \nu$;
\item[(b)] $\overline{F}( \tilde{\mu} \times \mu \times \nu)=\tilde{\mu} \times \mu \times \nu$;
\item[(c)] there exists a probability measure $\tilde{\nu}$ on $J_2$ such that the quadruplet of probability measures $(\mu,\nu,\tilde{\mu},\tilde{\nu})$ satisfies the detailed balance condition for $F$.
\end{enumerate}
\end{prop}

We now proceed to become more specific in our discussion. For the random polymer models of interest in \cite{CN}, as we will outline in more detail in Section \ref{sec2}, the maps $R$ that arise take the form
\begin{equation}\label{rab}
R_{(\a,\b)}(a,b,c)= \left( \frac{ac+ \a b + \b ab}{c}, \frac{ac+ \a b+ \b ab }{b} \right)
\end{equation}
for some $\a,\b\in\mathbb{R}$ satisfying $\max\{\a,\b\}>0$. One can readily check that a corresponding $F$ is given by
\begin{equation}\label{Fab}
F_{(\a,\b)}(x,y)=\left(\frac{x(y-\a)}{y+\b x}, \frac{y}{x} \right).
\end{equation}
(Note that the choice of $F$ is not unique, but the flexibility that it admits is essentially only a change of scale in the second coordinate of its output, and the above form will be convenient for us.) Consequently, if one can characterise detailed balance solutions for the maps $F_{(\a,\b)}$, then one obtains all the stationary measures of the polymer models. In fact, by scaling and symmetry, there are only essentially four choices of $(\a,\b)$ that are relevant in the setting of \cite{CN}:
\begin{itemize}
  \item $(\a,\b)=(0,1)$, which corresponds to the inverse-gamma(/log-gamma) model of \cite{S,Serr};
  \item $(\a,\b)=(1,0)$, which corresponds to the gamma(/strict-weak) model of \cite{ISS,OO};
  \item $(\a,\b)=(-1,1)$, which corresponds to the inverse-beta model of \cite{Thiery};
  \item $(\a,\b)=(1,-1)$, which corresponds to the beta model of \cite{BC,BRS}, see also Remark \ref{brwre} for comments about the connection with the random walk in a beta-distributed random environment ($\b$-RWRE) of \cite{BC}.
\end{itemize}
Moreover, after some additional rescaling, one can check that solving the detailed balance equation for the corresponding $F_{(\a,\b)}$ is equivalent to solving it for one of the following bijections on $(0,\infty)^2$:
\[F_{\Gam,\tBeta}(x,y):=\left(x+y, \frac{x}{y} \right),\]
which has inverse
\[F_{\Gam,\tBeta}^{-1}(x,y):=\left(\frac{xy}{1+y}, \frac{x}{1+y} \right),\]
and the involution
\[F_{\tBeta,\tBeta}(x,y):=\left(\frac{1+x}{y}, \frac{1+x+y}{xy} \right).\]
Precisely, we will explain how to relate $R_{(0,1)}$ to $F_{\Gam,\tBeta}$, $R_{(1,0)}$ to $F_{\Gam,\tBeta}^{-1}$, and both $R_{(-1,1)}$ and $R_{(1,-1)}$ to $F_{\tBeta,\tBeta}$. (See Proposition \ref{p31} below.) In the course of the study, it will also be convenient to consider $R_{(1,1)}$, which is essentially equivalent to $R_{(-1,1)}$ and also relates to $F_{\tBeta,\tBeta}$ (see Remark \ref{r11rem} for the connection between $R_{(1,1)}$ and $R_{(-1,1)}$, and \cite[Lemma 4.1]{CN} and Proposition \ref{p31} again for the link with $F_{\tBeta,\tBeta}$). Importantly, the solutions of the detailed balance condition for both $F_{\Gam,\tBeta}$ and $F_{\tBeta,\tBeta}$ are completely described in classical results, see Proposition \ref{gb} below. As indicated by the notation, the distributions appearing in these statements are either a gamma distribution or a beta prime distribution. (See Appendix \ref{appa} for a list of the various parametric families that appear in this article.) As a result, one is able to characterise the stationary measures for the original polymer models. We summarise the results in Figure \ref{fig1}; for precise claims, see Theorem \ref{t42} (and also \cite{CN}). As further background, we note that \cite{CN2} contains an application of the stationarity property of the above polymer models to the understanding of their fluctuation exponents.

Towards describing the zero-temperature limits of the above-mentioned polymer models, we define, for each $\varepsilon>0$, a map $S_\varepsilon:\mathbb{R}\rightarrow (0,\infty)$ by setting
\[S_{\epsilon} (x):= e^{-x/\varepsilon}.\]
Moreover, slightly abusing notation, we write $S_\varepsilon(x_1,x_2,\dots,x_n):=(S_\varepsilon(x_1),S_\varepsilon(x_2),\dots,S_\varepsilon(x_n))$. The zero-temperature limit of the map $R_{(\a,\b)}$ at \eqref{rab} is then given by the piecewise linear function $R^{\z}_{(\a,\b)}$, as defined formally by
\begin{equation}\label{ztl}
R^{\z}_{(\a,\b)}(a,b,c):=\lim_{\varepsilon\downarrow 0}S_{\varepsilon}^{-1}\circ R_{(\a,\b)}\circ S_{\varepsilon}(a,b,c).
\end{equation}
We highlight that whilst in this article we view the term `zero-temperature limit' purely as a mathematical operation, it is naturally motivated by statistical physics considerations. Indeed, the polymer models described above will typically incorporate a temperature parameter in their definition, and the $\varepsilon\rightarrow 0$ limit of \eqref{ztl} represents taking this parameter to 0 in the sense that, just as the map $R_{(\a,\b)}$ gives a recursion equation for the partition function in the relevant positive-temperature case, the map $R^{\z}_{(\a,\b)}$ describes an analogous recursion for the corresponding zero-temperature model. (See Section \ref{sec2} for some concrete examples of polymer models and their zero-temperature limits.) We further note that the limiting procedure described at \eqref{ztl} is commonly referred to as ultra-discretisation in the discrete integrable systems literature.

The domain upon which the limit at \eqref{ztl} is defined will vary depending on the particular choice of $(\a,\b)$, and, in the $(1,-1)$ case, we will make a change of scale before taking the zero-temperature limit (see Subsection \ref{ztldef} for details). For the special cases of polymer models listed above, these limits again correspond to models that have been studied hitherto. Specifically:
\begin{itemize}
  \item $(\a,\b)=(0,1)$ corresponds to directed last passage percolation (LPP) with exponential/geometric waiting times, which essentially appeared in \cite{J};
  \item $(\a,\b)=(1,0)$ corresponds to directed first passage percolation (FPP) with exponential/geometric waiting times, as studied in the language of queuing systems in \cite{DMO};
  \item $(\a,\b)=(-1,1)$ or $(\a,\b)=(1,1)$ corresponds to the Bernoulli-exponential/geometric polymer of \cite{TD};
  \item $(\a,\b)=(1,-1)$ corresponds to Bernoulli-exponential/geometric FPP, which was introduced as a zero-temperature limit of the $\b$-RWRE in \cite{BC}, and has also been called the river delta model (see \cite{BR}).
\end{itemize}
(See \cite[Sections V.B.1, V.B.2 and V.B.3]{Thiery} for further background on the first three models, respectively.) Unlike in the positive-temperature case, the maps $R^{\z}_{(\a,\b)}$ in these classes can not always be decomposed as per the picture at \eqref{picture} directly, even by taking the map $F^{\z}_{(\a,\b)}$ to be the zero-temperature limit of $F_{(\a,\b)}$. (See Remark \ref{rrr} and the preceding discussion below.) However, by making certain changes of scale corresponding to those of the positive-temperature case, one can nonetheless identify a family of both continuous and discrete stationary measures for each $R^{\z}_{(\a,\b)}$ by considering the detailed balance condition for one of the following bijections on $\mathbb{R}^2$:
\[F_{\E,\AL}(x,y):=\left((x\wedge y), x-y \right),\]
which has inverse
\[F_{\E,\AL}^{-1}(x,y):=\left(x+y-(0\wedge y), x-(0\wedge y) \right),\]
and the involution
\[F_{\AL,\AL}(x,y):=\left((0\wedge x) - y, (0\wedge x\wedge y)-x-y \right),\]
where we use the notation $x\wedge y:=\min\{x,y\}$ (and later also $x\vee y:=\max\{x,y\}$). Note that $F_{\Gam,\tBeta}^\z=F_{\E,\AL}$ and $F_{\tBeta,\tBeta}^\z=F_{\AL,\AL}$, with these zero-temperature limits being defined as at \eqref{ztl}. See Proposition \ref{p34} for the links between maps, and Theorem \ref{t45} for our main result on the stationary measures of zero-temperature random polymer models, which we believe is new for the $(\a,\b)=(1,-1)$ case in particular\footnote{In \cite[Section 1.3]{BR}, it was written that the stationary measure of Bernoulli-exponential FPP, which is the zero-temperature limit of the beta polymer, was computed in \cite{Thiery}. However, this seems to stem from a conflation of Bernoulli-exponential FPP and the Bernoulli-exponential polymer, which was studied in \cite{Thiery}. We highlight that these models are not the same; whereas the first is a zero-temperature limit of the beta polymer, the second is a zero-temperature limit of the inverse-beta polymer. Moreover, in \cite{Thiery}, it is stressed that the beta polymer is not present in the framework of that article, and `somehow lives in a different class'.}. (See Figure \ref{fig1} again.) Similarly to the positive-temperature case, as indicated by the notation, the distributions appearing in the latter result are either an exponential(/geometric) distribution or an asymmetric Laplace(/discrete asymmetric Laplace) distribution.

Two outstanding issues with the argument of the previous paragraph are as follows. Firstly, whilst we believe the detailed balance solutions that we present for $F_{\AL,\AL}$ are the only ones, we do not yet have a proof of this. (For $F_{\E,\AL}$, characterisation of solutions of the detailed balance equation is classical.) Secondly, since the links between the zero-temperature polymer setting and the basic bijections $F_{\E,\AL}$ and $F_{\AL,\AL}$ are in general not bijective, it is not possible to conclude that the stationary measures we identify for the polymer ones provide a complete characterisation of such. See Section \ref{sec6} for more detailed comments in this direction.

\begin{figure}
\begin{center}
\includegraphics[width=\textwidth]{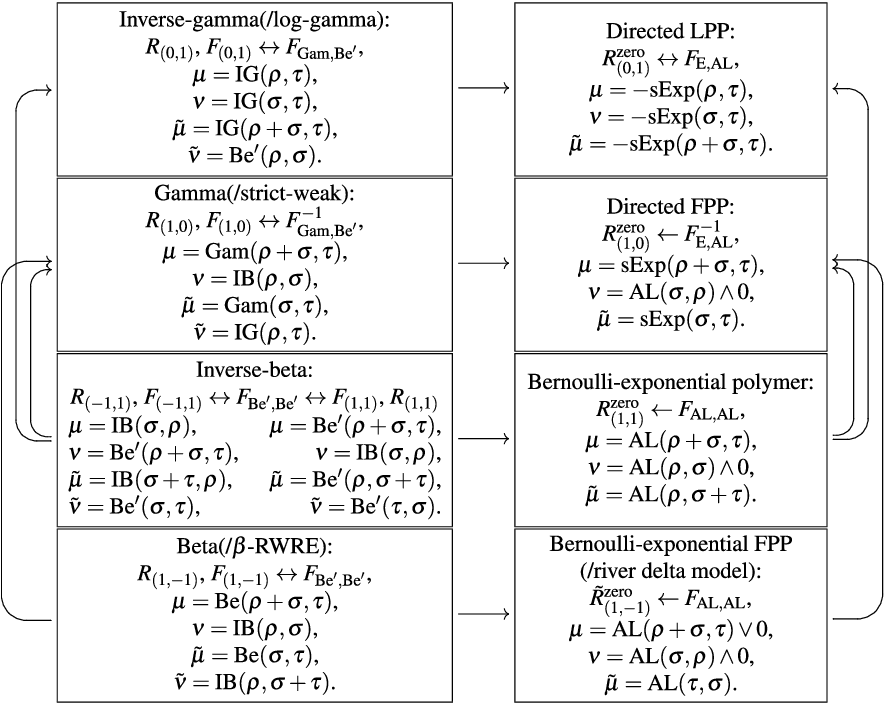}
\end{center}
\caption{The left-hand side shows the positive-temperature polymer models considered in this article, and the right-hand side shows the corresponding zero-temperature limits. The maps $R$ are those appearing in the original model (as determined by \eqref{rab} or \eqref{ztl}, except in the zero-temperature model with $(\alpha,\beta)=(1,-1)$, for which we make a change of variables before taking the limit at \eqref{ztl}). Moreover, in the positive-temperature case, the bijections $F$ are those corresponding to $R$ via a decomposition as at \eqref{picture}. In each case, we also note which of the basic bijections $F_{\Gam,\tBeta}$, $F_{\tBeta,\tBeta}$, $F_{\E,\AL}$ and $F_{\AL,\AL}$ the map $R$ can be built from (see Corollary \ref{c32} and Proposition \ref{p34}). The measures shown satisfy $R(\tilde{\mu}\times\mu\times\nu)=\mu\times\nu$, and also $F(\mu\times\nu)=\tilde{\mu}\times\tilde{\nu}$ in the positive-temperature case. See Theorems \ref{t42} and \ref{t45} for precise statements concerning the stationary measures of the random polymer models, including the additional discrete measures that arise in the zero-temperature case. Various relationships between models, as indicated by arrows, are discussed in Section \ref{sec5}; in all cases, these connect the three distributional parameters of the initial and limiting models, apart from in the links between zero-temperature models, which each see a reduction from three to two parameters (which means that not all the potential limiting distributions can be obtained by the procedures that we describe).}\label{fig1}
\end{figure}

As well as identifying the stationary measures of random polymer models, we investigate how these are related. To this end, we start by considering the connections between our basic maps, showing how each of $F_{\Gam,\tBeta}$ (and its inverse), $F_{\AL,\AL}$ and $F_{\E,\AL}$ (and its inverse) can be obtained through certain limiting operations from $F_{\tBeta,\tBeta}$, see Figure \ref{fig2}. (Cf. The centrality of the inverse-beta model in \cite[Figure 1]{Thiery}.) Furthermore, applying the same changes of variables, we can relate (some of the continuous) solutions of the detailed balance condition for these maps. Such observations naturally translate to the maps underlying the positive- and zero-temperature random polymer models considered in this article, and in some instances transfer to convergence of the partition function of the models in question. The latter links are shown by arrows on Figure \ref{fig1}, and discussed in more detail in Remarks \ref{ptptrem}, \ref{ptztrem} and \ref{ztrem}. Whilst we do give a new connection between the bijection $F_{(\a,\b)}$ and detailed balance solutions in the beta random polymer setting and those in the inverse-gamma setting (see Propositions \ref{p53}(d) and \ref{p55}(d)), we highlight that there remains a question concerning the relation of the partition function of these two models (and similarly in the zero-temperature setting).

The remainder of the article is organised as follows. Section \ref{sec2} contains precise definitions of the random polymer models studied in this article. In Section \ref{sec3}, we explain how the underlying recursion maps relate to one of the bijections $F_{\Gam,\tBeta}$,  $F_{\tBeta,\tBeta}$, $F_{\E,\AL}$ and $F_{\AL,\AL}$. In Section \ref{sec4}, we then apply the preceding results to obtain descriptions of stationary measures for the random polymer models. Following this, Section \ref{sec5} features a discussion of various relations between the detailed balance solutions identified for the basic bijections, and between the stationary measures for the random polymer maps $R_{(\a,\b)}$, as well as some remarks on connections between the partition functions of the corresponding polymer models. We conclude in Section \ref{sec6} with some conjectures and problems left open by this work. As noted above, we also include an appendix in which definitions of the various parametric families appearing in the discussion are collated.

\section{Polymer models}\label{sec2}

\subsection{Positive-temperature versions}

We start by considering a directed random polymer model with edge weights, which will naturally include a corresponding version with site weights. The key quantity of interest here will be the partition function
\[Z_{n,m}=\sum_{\pi : (0,0) \to (n,m)}\left\{\prod_{e \in \pi}Y_{e}\right\},\qquad \forall n,m\in\mathbb{Z}_+,\]
where $\mathbb{Z}_+:=\{0,1,\dots\}$, the sum is taken over up-right paths $\pi=(e_1,e_2,\dots, e_{n+m})$ from $(0,0)$ to $(n,m)$ on $\mathbb{Z}_+^2$, and the edge weights are defined by setting
\[Y_e:=\left\{\begin{array}{ll}
         u(X_{k,\ell}), & \mbox{if }e_i=((k-1,\ell), (k,\ell)),\\
         v(X_{k,\ell}), & \mbox{if }e_i=((k,\ell-1), (k,\ell)),
       \end{array}\right.\]
with $(X_{n,m})_{n,m \in \mathbb{Z}_+}$ being an array of i.i.d.\ $(0,\infty)$-valued random variables, and $u :(0,\infty) \to (0,\infty)$ and $v : (0,\infty) \to (0,\infty)$ being positive functions on $(0,\infty)$. For this model, we have the following recursive equation for the partition function:
\begin{equation}\label{RPe}
Z_{n,m}=u(X_{n,m})Z_{n-1,m}+v(X_{n,m})Z_{n,m-1}.
\end{equation}
Moreover, by setting
\[U_{n,m}:=Z_{n,m}/Z_{n-1,m}, \qquad V_{n,m}:=Z_{n,m}/Z_{n,m-1},\]
the recursive equation at \eqref{RPe} can be rewritten as at \eqref{R-eq}, or more specifically,
\[R_{RPe}(X_{n,m}, U_{n,m-1},V_{n-1,m})=(U_{n,m}, V_{n,m}),\]
where
\[R_{RPe}(a,b,c)=\left(u(a)+v(a)\frac{b}{c}, u(a)\frac{c}{b}+v(a) \right).\]
Note that if $u(x)=v(x)=x$, then the model reduces to a directed random polymer model with site weights, namely the partition function can be written
\[Z_{n,m}=\sum_{\pi : (0,0) \to (n,m)}\left\{\prod_{(k,\ell)\in \pi\backslash\{(0,0)\}}X_{k,\ell}\right\}.\]
In this case, $R_{RPe}$ simplifies to
\[R_{RPs}(a,b,c)=\left(\frac{a(b+c)}{c}, \frac{a(b+c)}{b} \right).\]

We next explain formally how the map $R_{RPe}$ can be decomposed as at \eqref{picture}. First, let $G_{RPe} :(0,\infty)^2 \to (0,\infty)^2$ be given by
\begin{equation}\label{gmap}
G_{RPe}(x,y):=\left(u(x)+\frac{v(x)}{y}, u(x)y+v(x)\right),
\end{equation}
and observe $R_{RPe}(a,b,c)=G_{RPe}(a, \frac{c}{b})$. Now, if $H_y(x):=u(x)y+v(x)$ is invertible on a suitable interval for each $y$, then one can check that the function $F_{RPe}$ defined by
\[F_{RPe}(x,y):=\left(H^{-1}_{\frac{y}{x}}(y), \frac{y}{x}\right)\]
satisfies $G_{RPe} \circ F_{RPe} (x,y)=F_{RPe} \circ G_{RPe}(x,y)=(x,y)$ on a certain subset of $(0,\infty)^2$. Consequently, we have that \eqref{FR} holds for this choice of $F$, i.e.\ $R_{RPe}=\overline{F}_{RPe}^{(2,3)}$. By Proposition \ref{twothree}, to find solutions of \eqref{Rinv}, it is thus enough to solve the detailed balance equation for $F_{RPe}$.

In \cite{CN}, the authors carefully study the detailed balance problem for $F_{RPe}$ under the conditions that $u(x)=x$, $v$ is suitably differentiable, $H_y$ satisfies the necessary invertibility, and some further technical assumptions on the measures. In this setting, the authors establish that there exists a solution of \eqref{db} only when $v(x)=\alpha+\beta x$ for some $\alpha,\beta \in \R$ such that $\max\{\alpha,\beta\} >0$. Note that in these cases, the maps $R_{RPe}$ and $F_{RPe}$ are given by the $R_{(\a,\b)}$ and $F_{(\a,\b)}$ appearing at \eqref{rab} and \eqref{Fab}, respectively. Moreover, in \cite{CN}, the solutions of \eqref{db} are characterised explicitly. By making simple changes of variables, it is shown that these can be reduced to the four cases listed in the introduction, with $(\alpha,\beta)$ being given by $(0,1)$, $(1,0)$, $(-1,1)$ or $(1,-1)$. (The map $R_{RPs}$ is equal to $R_{(0,1)}$.) For convenience, we summarise the maps $R_{(\a,\b)}$, $F_{(\a,\b)}$ and the domains/co-domains that we consider in Table \ref{maptable}.

\begin{table}
\begin{center}
\begin{tabular}{m{1.2cm}|m{3.5cm}|m{2.5cm}|m{5cm}}
  $(\a,\b)$ &  $R_{(\a,\b)}(a,b,c)$ & $F_{(\a,\b)}(x,y)$ &\vphantom{a}$I_1\times I_2\rightarrow J_1\times J_2$\\[2ex] \hline
  $(0,1)$ & $\left(\frac{a(b+c)}{c}, \frac{a(b+c)}{b}\right)$& $\left(\frac{xy}{x+y}, \frac{y}{x}\right)$ &\vphantom{a}$(0,\infty)^2 \to (0,\infty)^2$\\[2ex]
  \hline
  $(1,0)$ & $\left(\frac{ac +b }{c}, \frac{ac+b}{b}\right)$ & $\left(\frac{x(y-1)}{y}, \frac{y}{x}\right)$ &\vphantom{a} $(0,\infty)\times (1,\infty) \to (0,\infty)^2$\\[2ex]\hline
 $(-1,1)$ & $\left(\frac{ac-b+ab}{c}, \frac{ac-b+ab}{b}\right)$ & $\left(\frac{x(y+1)}{x+y}, \frac{y}{x}\right)$ & \vphantom{a}$(1,\infty) \times (0,\infty) \to (1,\infty) \times (0,\infty)$\\[2ex] \hline
$(1,1)$& $\left(\frac{ac+b+ab}{c}, \frac{ac+b+ab}{b}\right)$ & $\left(\frac{x(y-1)}{x+y}, \frac{y}{x}\right)$ &\vphantom{a}$(0,\infty) \times (1,\infty) \to  (0,\infty)^2$ \\[2ex]\hline
$(1,-1)$ & $\left(\frac{ac+b-ab}{c}, \frac{ac+b-ab}{b}\right)$ & $\left(\frac{x(y-1)}{y-x}, \frac{y}{x}\right)$ &\vphantom{a}$(0,1) \times (1,\infty) \to (0,1) \times (1,\infty)$\\[2ex]
\end{tabular}
\end{center}
\caption{The maps $R_{(\a,\b)}$, $F_{(\a,\b)}$ and the domains/co-domains studied in this article. (For the domains/co-domains, we use the convention from the introduction that $R:J_1\times I_1\times I_2 \rightarrow I_1\times I_2$ and $F:I_1\times I_2\rightarrow J_1\times J_2$.)}\label{maptable}
\end{table}

\begin{rem}\label{brwre}  The random walk in a beta-distributed random environment of \cite{BC} can also be described within the above framework. Specifically, let $(\hat{X}_{n,m})_{n \in\Z_+, m \in \Z}$ be i.i.d.\ $(0,1)$-valued random variables. Let $S_n$ be the position of a random walk in a random environment given by
\[\mathbf{P}_{\hat{X}}(S_{n+1}=m+1|S_n=m)=\hat{X}_{n,m}=1-\mathbf{P}_{\hat{X}}(S_{n+1}=m-1|S_n=m),\]
where $\mathbf{P}_{\hat{X}}$ is the distribution of the random walk given the environment $(\hat{X}_{n,m})_{n \in\Z_+, m \in \Z}$. If we define
\[\hat{Z}_{n,m}:=\mathbf{P}_{\hat{X}}(S_m \ge m-2n+2),\]
then one can check that
\[\hat{Z}_{n,m}:=\hat{X}_{n,m}\hat{Z}_{n,m-1}+(1-\hat{X}_{n,m})\hat{Z}_{n-1,m-1}\]
Making the change of variables $\hat{Z}_{n,m}\equiv Z_{n,m-n}$ and $\hat{X}_{n,m}\equiv X_{n,m-n}$, we arrive at \eqref{RPe} with $u(x)=1-x$ and $v(x)=x$.
\end{rem}

\subsection{Zero-temperature limits}\label{ztldef}

We next consider the zero-temperature directed random polymer models with edge weights, which includes directed last/first passage percolation with site weights. In this case, the partition function is given by
\[Z_{n,m}=\min_{\pi : (0,0) \to (n,m)}\left\{\sum_{e \in \pi}Y_e\right\},\qquad \forall n,m\in\mathbb{Z}_+,\]
with the minimum being taken over up-right paths $\pi$ from $(0,0)$ to $(n,m)$ on $\mathbb{Z}_+^2$, and the weights $Y_e$ being defined from i.i.d.\ random variables $(X_{n,m})_{n,m \in \mathbb{Z}_+}$ similarly to the previous subsection. The difference is that now we allow the latter random variables to be $\R$-valued, and we suppose $u :\R \to \R$ and $v : \R \to \R$ are functions on $\R$. The recursive equation for the partition function now becomes:
\begin{equation}\label{zRPe}
Z_{n,m}=\min\{ Z_{n-1,m} +u(X_{n,m}), Z_{n,m-1}+v(X_{n,m})\}.
\end{equation}
Moreover, by setting
\[U_{n,m}:=Z_{n,m}-Z_{n-1,m}, \qquad V_{n,m}:=Z_{n,m}-Z_{n,m-1},\]
the recursive equation \eqref{zRPe} can be rewritten as
\[R_{RPe}^{\z}(X_{n,m}, U_{n,m-1},V_{n-1,m})=(U_{n,m}, V_{n,m}),\]
where
\[R_{RPe}^\z(a,b,c)=\left(\min\{u(a),b-c+v(a)\}, \min\{u(a)+c-b,v(a)\}\right).\]
We note that if the above $u$ and $v$ are the zero-temperature limits of the $u$ and $v$ defining $R_{RPe}$ (i.e.\ analogously to \eqref{ztl}, are given by $\varepsilon\rightarrow0$ limits of $S_\varepsilon^{-1}\circ u\circ S_\varepsilon$ and $S_\varepsilon^{-1}\circ v\circ S_\varepsilon$ for the positive temperature version of $u$ and $v$, respectively), then the function $R_{RPe}^\z$ is the zero-temperature limit of $R_{RPe}$ in the sense of \eqref{ztl}. As noted in the introduction, the domains upon which we can do this will vary on the specific form of the functions $u$ and $v$. In particular, amongst the special cases of $R_{(\a,\b)}$ of interest here, we have:
\begin{itemize}
\item $R^{\z}_{(0,1)}:\R^3 \to \R^2$ is given by taking $u(a)=v(a)=a$, i.e.\
\[R^{\z}_{(0,1)}(a,b,c)= \left( \min\{a, a+b-c\}, \min\{ a+c -b,a \} \right);\]
\item  $R^\z_{(1,0)} : \R \times \R \times (-\infty,0] \to  \R \times (-\infty,0]$ is given by taking $u(a)=a$, $v(a)=0$, i.e.\
\[R^{\z}_{(1,0)}(a,b,c)= \left( \min\{a,b-c \}, \min\{a+c-b,0\}\right);\]
\item $R^{\z}_{(1,1)} : \R \times \R \times (-\infty,0]  \to \R \times (-\infty,0] $ is given by $u(a)=a$, $v(a)=\min\{a,0\}$, i.e.\
\[R^{\z}_{(1,1)}(a,b,c)=\left(\min\{a, \min\{a,0\}+b-c\}, \min\{a+c-b, \min\{a,0\} \}\right).\]
\end{itemize}
Note that we choose to consider the zero-temperature limit of $R_{(1,1)}$ rather than the equivalent map $R_{(-1,1)}$ to avoid the negative terms that appear in the latter map, since these lead to complications in defining the zero-temperature limit; in particular, one has to be more careful about the domain upon which the limit at \eqref{ztl} is defined. To deal with the same issue in the remaining case, that is $(\a,\b)=(1,-1)$, it will be convenient to make a change of variables before taking a limit. In particular, we take the zero-temperature limit of the map $\tilde{R}_{(1,-1)}:(0,\infty) \times (0,1) \times (1,\infty) \to (0,1) \times (1,\infty)$ given by
\begin{equation}\label{tilder}
\tilde{R}_{(1,-1)}(a,b,c)=\left(\frac{1}{1+a}+\frac{ab}{(1+a)c}, \frac{c}{(1+a)b}+ \frac{a}{1+a} \right)
\end{equation}
to obtain the following:
\begin{itemize}
\item $\tilde{R}^{\z}_{(1,-1)}: \R \times [0,\infty) \times (-\infty,0] \to [0,\infty) \times (-\infty,0]$ is given by the expression for $R^\z_{RPe}$ above with $u(a)=-\min\{0,a\}$, $v(a)=\max\{a,0\}$, i.e.\
\begin{align*}
&\qquad\qquad\tilde{R}^{\z}_{(1,-1)}(a,b,c)\\
&\qquad\qquad=\left(\min \{-\min\{0,a\} , \max\{a,0\} +b-c\}, \min\{-\min\{0,a\}+c-b,\max\{a,0\} \}\right)\\
&\qquad\qquad=\left(\min \{-\min\{0,a\} , b-c\}, \min\{-\min\{0,a\}+c-b,0 \}\right).
\end{align*}
\end{itemize}
(See \eqref{tildereq} for the precise connection between $\tilde{R}_{(1,-1)}$ and ${R}_{(1,-1)}$, and Remark \ref{Rtilderem} for discussion of a link between the zero-temperature versions of these maps.)

Now, suppose we want to decompose the maps $R^\z_{(\a,\b)}$ as at \eqref{picture} by following the argument of the previous subsection. Analogously to \eqref{gmap}, one might then start by considering the map $G :\R^2 \to \R^2$ given by
\[G(x,y):=\left( \min\{ u(x), v(x)- y\}, \min\{ u(x)+y, v(x) \} \right),\]
so that $R^\z_{RPe}(a,b,c)=G(a, c-b)$. If $H_y(x):=\min\{ u(x)+y, v(x) \}$ was invertible on a suitable interval for each $y$, then $F(x,y):=(H^{-1}_{y-x}(y), y-x)$ would satisfy $G \circ F (x,y)=F \circ G(x,y)=(x,y)$ and $\overline{F}^{(2,3)}=R^\z_{RPe}$. Characterising the solutions of \eqref{Rinv} for $R^\z_{RPe}$ would then reduce to characterising the solutions of the detailed balance equation for $F$. However, in the particular examples of $R^\z_{(\a,\b)}$ considered here, the desired invertibility hardly holds, see Remark \ref{rrr}. To handle this issue, as is detailed in the next section, we make changes of variables to write the maps $R^\z_{(\a,\b)}$ in terms of the more convenient, invertible maps $F_{\E,\AL}$ and $F_{\AL,\AL}$.

\begin{rem}\label{rrr}
With regards to the four zero-temperature polymer models considered here, the invertibility of the map $x\mapsto H_y(x)$ only holds in general for $R^{\z}_{(0,1)}$. In this case, one can compute that
\[R^{\z}_{(0,1)}=\overline{{F}^\z_{(0,1)}}^{(2,3)},\]
where ${F}^\z_{(0,1)}:\mathbb{R}^2\rightarrow\mathbb{R}^2$, as defined by
\begin{equation}\label{fz01}
{F}^\z_{(0,1)}(x,y)=\left(x\vee y,y-x\right),
\end{equation}
is the zero-temperature limit of $F_{(0,1)}$.
\end{rem}

\begin{rem}
The case when $u(a)=v(a)=a$, i.e.\ $R^{\z}_{(0,1)}$, clearly corresponds to a directed first passage percolation model with site weights. Moreover, when $u(x)=v(x)=-x$, then by letting $\tilde{Z}_{n,m}:=-Z_{n,m}$ and $\tilde{X}_{n,m}:=-X_{n,m}$, we have
\[\tilde{Z}_{n,m}=\max_{\pi : (0,0) \to (n,m)}\left\{\sum_{(k,\ell) \in \pi\backslash\{(0,0)\}}\tilde{X}_{k,\ell}\right\},\]
which is a directed last passage percolation partition function that satisfies
\begin{equation}\label{LLP}
\tilde{Z}_{n,m}=\max\{ \tilde{Z}_{n-1,m}, \tilde{Z}_{n,m-1}\}+\tilde{X}_{n,m}.
\end{equation}
By letting
\[\tilde{U}_{n,m}:=\tilde{Z}_{n,m}-\tilde{Z}_{n-1,m}, \qquad \tilde{V}_{n,m}:=\tilde{Z}_{n,m}-\tilde{Z}_{n,m-1},\]
the recursive equation \eqref{LLP} can be rewritten
\[R_{DLPP}(\tilde{X}_{n,m}, \tilde{U}_{n,m-1},\tilde{V}_{n-1,m})=(\tilde{U}_{n,m}, \tilde{V}_{n,m}),\]
where
\[R_{DLPP}(a,b,c)=(a+b-\min\{b,c\}, a+c-\min\{b,c\})=-R^{\z}_{(0,1)}(-a,-b,-c),\]
which is essentially the same map as for the case of first passage percolation.
\end{rem}

\section{Reduction to basic bijections}\label{sec3}

As we commented in the introduction, there are four basic bijections that appear in this study: $F_{\Gam,\tBeta}$, $F_{\tBeta,\tBeta}$, $F_{\E,\AL}$ and $F_{\AL,\AL}$. The aim of this section is to explain how the bijections $F_{(0,1)}$, $F_{(1,0)}$, $F_{(-1,1)}$ (or $F_{(1,1)}$) and $F_{(1,-1)}$ arising in the positive-temperature polymer models can each be transformed to one of $F_{\Gam,\tBeta}$, $F_{\Gam,\tBeta}^{-1}$ or $F_{\tBeta,\tBeta}$ by a bijective transformation of the state space. We will further present corresponding links between the zero-temperature maps $R^\z_{(0,1)}$, $R^\z_{(1,0)}$, $R^\z_{(1,1)}$ and $\tilde{R}^\z_{(1,-1)}$ and the bijections $F_{\E,\AL}$, $F_{\E,\AL}^{-1}$ or $F_{\AL,\AL}$.

\subsection{Positive-temperature versions}

To describe the change of scales of interest in the positive-temperature case, we introduce a number of bijections. Firstly, we let $I_d$ be the identity map on $\mathbb{R}$; clearly this is a bijection on any subset of $\mathbb{R}$. Next, define $I: (0,\infty) \to (0,\infty)$ by setting
\[I(x)=x^{-1}.\]
We note that this is also a bijection between $(0,1)$ and $(1,\infty)$, and vice versa. Moreover, let $Q : (0,1) \to (0,\infty)$ be the bijection given by
\[Q(x)=x^{-1}-1;\]
the corresponding inverse is given by
\[Q^{-1}(x)=\frac{1}{1+x}.\]
We note that $J := Q^{-1} \circ I \circ Q : (0,1) \to (0,1)$ satisfies $J(x)=1-x$. Finally, we suppose $\pi: (0,\infty)^2 \to  (0,\infty)^2$ is the coordinate transposition map, i.e.\
\[\pi(x,y)=(y,x).\]
Given this notation, it is an elementary exercise to check the following proposition. In the statement of this, we write $f\times g$ for the product of maps $f$ and $g$, i.e.\ the map given by $(f\times g)(x,y):=(f(x),g(y))$. We recall that the maps $F_{\Gam,\tBeta}$ and $F_{\tBeta,\tBeta}$ are considered to be bijections on $(0,\infty)^2$.

\begin{prop}[{cf.\cite{CN}}]\label{p31}
(a) The bijection $F_{(0,1)}:(0,\infty)^2\rightarrow(0,\infty)^2$ satisfies
\[(I \times I_d) \circ F_{(0,1)}  \circ (I \times I) =F_{\Gam,\tBeta}.\]
(b) The bijection $F_{(1,0)}:(0,\infty)\times(1,\infty)\rightarrow(0,\infty)^2$ satisfies
\[\pi \circ (I_d \times I) \circ F_{(1,0)} \circ (I_d \times (I \circ Q^{-1} \circ I)) =F_{\Gam,\tBeta}^{-1}.\]
(c)(i) The bijection $F_{(-1,1)}:(1,\infty)\times(0,\infty)\rightarrow(1,\infty)\times(0,\infty)$ satisfies
\[\pi \circ ( (I \circ Q \circ I) \times I )\circ F_{(-1,1)} \circ ((I \circ Q^{-1}) \times I_d) =F_{\tBeta,\tBeta}.\]
(ii) The bijection $F_{(1,1)}:(0,\infty)\times(1,\infty)\rightarrow(0,\infty)^2$ satisfies
\[ \pi \circ ( I \times I_d )\circ F_{(1,1)} \circ ( I_d \times (I \circ Q^{-1})) \circ \pi =F_{\tBeta,\tBeta}.\]
(d)  The bijection $F_{(1,-1)}:(0,1)\times(1,\infty)\rightarrow(0,1)\times(1,\infty)$ satisfies
\[( Q \times (Q \circ I) )\circ F_{(1,-1)} \circ ((Q^{-1} \circ I) \times (I \circ Q^{-1} \circ I) ) \circ \pi =F_{\tBeta,\tBeta}.\]
\end{prop}

\begin{rem}\label{f11rem} Applying the relation $F_{(1,1)}=((Q \circ I) \times I) \circ F_{(-1,1)} \circ \pi$, parts (c)(i) and (c)(ii) of the above result are readily deduced from each other.
\end{rem}

For understanding the zero-temperature versions of the models, the following straightforward corollary of Proposition \ref{p31} will be useful. For establishing the result, we note that the map $\tilde{R}_{(1,-1)}$ defined at \eqref{tilder} satisfies
\begin{equation}\label{tildereq}
\tilde{R}_{(1,-1)}\equiv R_{(1,-1)} \circ (Q^{-1} \times I_d \times I_d )
\end{equation}
on the relevant domain. Moreover, it is maybe helpful to have in mind that for bijections $F$ and $G$ whose domains and codomains are subsets of $\mathbb{R}^2$, if $(f_1\times g_1)\circ G\circ (f_2\times g_2)=F$, where $f_1,g_1,f_2,g_2$ are real-valued functions defined on appropriate subsets of $\mathbb{R}$ with $f_1$ a bijection and $g_1$ an injection, then \[\overline{G}^{(2,3)}\circ \left(f_1^{-1}\times f_2\times g_2\right)=\left(f_2\times g_2\right)\circ \overline{F}^{(2,3)}.\]

\begin{cor}\label{c32}
(a) The map $R_{(0,1)}:(0,\infty)^3\rightarrow(0,\infty)^2$ satisfies
\[R_{(0,1)} \circ (I \times I \times I) =(I \times I) \circ \overline{F_{\Gam,\tBeta}}^{(2,3)}.\]
(b) The map $R_{(1,0)}:(0,\infty)^2\times(1,\infty)\rightarrow(0,\infty)\times(1,\infty)$ satisfies
\[R_{(1,0)} \circ (I_d \times I_d \times (I \circ Q^{-1} \circ I)) = (I_d \times (I \circ Q^{-1} \circ I)) \circ \overline{ (\pi  \circ  F_{\Gam,\tBeta}^{-1} )}^{(2,3)}.\]
(c)(i) The map $R_{(-1,1)}:(1,\infty)^2\times(0,\infty)\rightarrow(1,\infty)\times(0,\infty)$ satisfies
\[R_{(-1,1)} \circ ((I \circ Q^{-1} \circ I)\times (I \circ Q^{-1}) \times I_d) = ((I \circ Q^{-1}) \times I_d) \circ  \overline{ (\pi\circ F_{\tBeta,\tBeta})}^{(2,3)}.\]
(ii) The map $R_{(1,1)}:(0,\infty)^2\times(1,\infty)\rightarrow(0,\infty)\times(1,\infty)$ satisfies
\[R_{(1,1)} \circ ( I\times  I_d  \times (I \circ Q^{-1})) = (  I_d  \times (I \circ Q^{-1})) \circ  \overline{(\pi \circ F_{\tBeta,\tBeta} \circ \pi) }^{(2,3)}.\]
(d)  The map $\tilde{R}_{(1,-1)}:(0,\infty)\times(0,1)\times(1,\infty)\rightarrow(0,1)\times(1,\infty)$ satisfies
\[\tilde{R}_{(1,-1)} \circ (I_d \times (Q^{-1} \circ I) \times (I \circ Q^{-1} \circ I) )= ((Q^{-1} \circ I) \times (I \circ Q^{-1} \circ I)) \circ  \overline{ (F_{\tBeta,\tBeta} \circ \pi)}^{(2,3)}.\]
\end{cor}

\begin{rem}\label{r11rem} From the relation between $F_{(1,1)}$ and $F_{(-1,1)}$ that was applied in Remark \ref{f11rem}, one deduces that $R_{(1,1)}=\pi\circ R_{(-1,1)} \circ (I\circ Q^{-1}\times \pi)$. Hence, characterising solutions of \eqref{Rinv} for either $R_{(1,1)}$ or $R_{(-1,1)}$ is enough to characterise solutions of \eqref{Rinv} for the other.
\end{rem}

\subsection{Zero-temperature limits}

To transfer the corresponding changes of scale to the zero-temperature regime, we first observe that, by a simple calculation,
\[I_d^\z(x)=x,\qquad I^{\z}(x)= -x, \qquad (Q^{-1})^{\z}(x)=-\min \{x,0\},\]
where we define the zero-temperature limit similarly to \eqref{ztl} in each case. Importantly, whilst the first and second of these maps are bijections on $\mathbb{R}$ (or indeed between suitable subsets of $\mathbb{R}$), the third map loses injectivity on the positive real axis. Moreover, for the third map, the image of the map is not an open interval, but a closed one, since $(Q^{-1})^{\z}(\R)= [0,\infty)$. In particular, $(Q^{-1})^{\z}([0,\infty))= \{0\}$, which has as a consequence that the image under $(Q^{-1})^{\z}$ of an absolutely continuous probability measure on $\R$ may have an atom. More precisely, one can check that $S_{\epsilon}^{-1} \circ (Q^{-1}) \circ S_{\epsilon} (0) =\epsilon \log 2$, and further
\[ S_{\epsilon}^{-1} \circ (Q^{-1}) \circ S_{\epsilon} ( (0,\infty))= (0, \epsilon \log 2),  \qquad S_{\epsilon}^{-1} \circ (Q^{-1}) \circ S_{\epsilon} ( (-\infty,0))= (\epsilon \log 2, \infty).\]
Hence, even though $0$ is not in the image of $ S_{\epsilon}^{-1} \circ (Q^{-1}) \circ S_{\epsilon} $ for any $\epsilon >0$, the image of $(Q^{-1})^{\z}$ includes $0$ and the inverse image of $\{0\}$ is $[0,\infty)$. As we will describe in Subsection \ref{s42}, this will have implications when it comes to characterising stationary solutions of the random polymer models of interest. Nonetheless, since the operation of composition is consistent with taking zero-temperature limits (in the sense that, when all the relevant limits are defined, $(f\circ g)^\z=f^\z\circ g^\z$), we readily read off the following result from Corollary \ref{c32}. Of course, it is also possible to check the relations by direct computation.

\begin{prop}\label{p34}
(a) The map $R^{\z}_{(0,1)}:\R^3 \to \R^2$ satisfies
\[R^\z_{(0,1)} \circ (I^\z \times I^\z \times I^\z) =(I^\z \times I^\z) \circ \overline{F_{\E,\AL}}^{(2,3)}.\]
(b) The map $R^\z_{(1,0)} : \R \times \R \times (-\infty,0] \to  \R \times (-\infty,0]$ satisfies
\begin{align*}
&R^\z_{(1,0)} \circ (I^\z_d \times I^\z_d \times (I^\z \circ (Q^{-1})^\z \circ I^\z))\\
&\qquad= (I^\z_d \times (I^\z \circ (Q^{-1})^\z \circ I^\z)) \circ \overline{ (\pi  \circ  F_{\E,\AL}^{-1} )}^{(2,3)}.
\end{align*}
(c) The map $R^{\z}_{(1,1)} : \R \times \R \times (-\infty,0]  \to \R \times (-\infty,0] $ satisfies
\[R^\z_{(1,1)} \circ ( I^\z\times  I^\z_d  \times (I^\z \circ (Q^{-1})^\z)) = (  I^\z_d  \times (I^\z \circ (Q^{-1})^\z)) \circ  \overline{(\pi \circ F_{\AL,\AL} \circ \pi) }^{(2,3)}.\]
(d) The map $\tilde{R}^{\z}_{(1,-1)}: \R \times [0,\infty) \times (-\infty,0] \to [0,\infty) \times (-\infty,0]$ satisfies
\begin{align*}
&\tilde{R}^\z_{(1,-1)} \circ (I^\z_d \times ((Q^{-1})^\z \circ I^\z) \times (I^\z \circ (Q^{-1})^\z \circ I^\z) )\\
&\qquad= (((Q^{-1})^\z \circ I^\z) \times (I^\z \circ (Q^{-1})^\z \circ I^\z)) \circ  \overline{ (F_{\AL,\AL} \circ \pi)}^{(2,3)}.
\end{align*}
\end{prop}

\begin{rem}
Taking zero-temperature limits of Proposition \ref{p31}(a), we find that the map ${F}^\z_{(0,1)}:\mathbb{R}^2\rightarrow\mathbb{R}^2$ defined at \eqref{fz01} satisfies
\begin{equation}\label{rrr2}
(I^\z \times I^\z_d) \circ F^\z_{(0,1)}  \circ (I^\z \times I^\z) =F_{\E,\AL}.
\end{equation}
\end{rem}

\begin{rem}\label{Rtilderem}
Noting that $(Q^{-1})^\z (\mathbb{R})=[0,\infty)$, and that $R_{(1,-1)}^\z$ is defined on $[0,\infty)\times [0,\infty)\times (-\infty,0]$, one can deduce from \eqref{tildereq} that, on $\mathbb{R}\times [0,\infty)\times (-\infty,0]$,
\[\tilde{R}^\z_{(1,-1)}\equiv R^\z_{(1,-1)} \circ ((Q^{-1})^\z\times I_d^\z \times I_d^\z),\]
where $R^\z_{(1,-1)}$ is the zero-temperature limit of the map $R_{(1,-1)}$. However, unlike in the positive-temperature case (see \eqref{tildereq}), the function linking $\tilde{R}^\z_{(1,-1)}$ and $R^\z_{(1,-1)}$ is non-bijective.
\end{rem}

\section{Detailed balance solutions and stationary measures}\label{sec4}

In the probability literature, the following kind of problem is classical. Let $X$ and $Y$ be independent $(0,\infty)$- (or $\mathbb{R}$-)valued random variables such that the distribution of $(X,Y)$ is non-trivial. Given a bijection $F$ from $(0,\infty)^2$ to itself (or $\mathbb{R}^2$ to itself), define
\[(U,V):=F(X,Y).\]
For what distributions of $X$ and $Y$ are $U$ and $V$ independent? Of course, one might equivalently ask: what are the non-trivial solutions of the detailed balance condition for $F$? Work on such questions originate in the work of Kac, who considered the case where $F:\mathbb{R}^2\rightarrow \mathbb{R}^2$ is given by $F(x,y)=(x+y,x-y)$, which turns out to characterise the normal distribution \cite{Kac}. (See also related discussion in \cite[Subsection 8.2]{CS2}.) In Subsection \ref{s41}, we recall how known results concerning the solutions of the detailed balance condition for $F_{\Gam,\tBeta}$ and $F_{\tBeta,\tBeta}$ can be used to characterise solutions of the detailed balance conditions for $F_{(0,1)}$, $F_{(1,0)}$, $F_{(-1,1)}$ (or $F_{(1,1)}$) and $F_{(1,-1)}$, and also stationary measures for the corresponding $R_{(\a,\b)}$. Moreover, in Subsection \ref{s42}, we explain how solutions of the detailed balance condition for $F_{\E,\AL}$ and $F_{\AL,\AL}$ yield stationary measures for  $R^\z_{(0,1)}$, $R^\z_{(1,0)}$, $R^\z_{(1,1)}$ and $\tilde{R}^\z_{(1,-1)}$, although we do not provide a complete characterisation of stationary measures in these zero-temperature cases. We recall that definitions  of the distributions appearing in our results can be found in Appendix \ref{appa}.

\subsection{Positive-temperature versions}\label{s41}

We start by setting out the solutions of the detailed balance condition for $F_{\Gam,\tBeta}$ and $F_{\tBeta,\tBeta}$, which are direct consequences of classical results.

\begin{prop}[{\cite{L,SW}}]\label{gb}
(a) The only non-trivial solutions of the detailed balance condition for the map $F_{\Gam,\tBeta}:(0,\infty)^2\rightarrow(0,\infty)^2$ are given by
\[\mu = \Gam(\rho,\tau),\qquad \nu = \Gam(\sigma,\tau),\qquad\tilde{\mu}= \Gam(\rho+\sigma,\tau),\qquad \tilde{\nu} = \tBeta(\rho,\sigma),\]
where $\rho,\sigma,\tau>0$.\\
(b) The only non-trivial solutions of the detailed balance condition for $F_{\tBeta,\tBeta}:(0,\infty)^2\rightarrow(0,\infty)^2$ are given by
\[\mu = \tBeta(\rho,\sigma),\qquad \nu = \tBeta(\rho+\sigma,\tau),\qquad \tilde{\mu} = \tBeta(\tau,\sigma),\qquad \tilde{\nu} = \tBeta(\sigma+\tau,\rho),\]
where $\rho,\sigma,\tau>0$.
\end{prop}
\begin{proof} The proof of the form of $\mu$ and $\nu$ in (a) is contained in \cite{L}. By considering the characteristic function of $F_{\tBeta,\tBeta}(X,Y)$ where $(X,Y)\sim{\mu}\times {\nu}$ as on the final page of \cite{L}, the form of the distributions $\tilde{\mu}$ and $\tilde{\nu}$ is readily deduced from this. As for part (b), on transforming $(0,\infty)$ to $(0,1)$ by the bijection $x\mapsto x/(1+x)$, this is equivalent to the characterisation of the beta distribution given in \cite{SW}.
\end{proof}

Combining the above result with Proposition \ref{p31} (noting that all the transformations in the statement of the latter result are bijections) readily yields the following. In checking this, it is useful to note the following relations between distributions:
\[I(\mathrm{Gam}(\rho,\sigma))=\mathrm{IG}(\rho,\sigma),\qquad I(\mathrm{Be}(\rho,\sigma))=\mathrm{IB}(\rho,\sigma),\]
\[I\circ Q(\mathrm{Be}(\rho,\sigma))=\tBeta(\rho,\sigma),\qquad I(\tBeta(\rho,\sigma))=\tBeta(\sigma,\rho).\]
(We do not need to use it here, but it also informative to note that $\mathrm{IB}(\rho,\sigma)=1+\tBeta(\sigma,\rho)$.)

\begin{thm}\label{t42}
(a) The only non-trivial solutions of the detailed balance condition for $F_{(0,1)}:(0,\infty)^2\rightarrow(0,\infty)^2$ are given by
\[\mu = \mathrm{IG}(\rho,\tau),\qquad \nu = \mathrm{IG}(\sigma,\tau),\qquad\tilde{\mu}= \mathrm{IG}(\rho+\sigma,\tau),\qquad \tilde{\nu} = \tBeta(\rho,\sigma).\]
(b) The only non-trivial solutions of the detailed balance condition for $F_{(1,0)}:(0,\infty)\times(1,\infty)\rightarrow(0,\infty)^2$ are given by
\[{\mu}= \Gam(\rho+\sigma,\tau),\qquad {\nu} = \mathrm{IB}(\rho,\sigma),\qquad\tilde{\mu} = \Gam(\sigma,\tau),\qquad \tilde{\nu} = \mathrm{IG}(\rho,\tau).\]
(c)(i) The only non-trivial solutions of the detailed balance condition for $F_{(-1,1)}:(1,\infty)\times(0,\infty)\rightarrow(1,\infty)\times(0,\infty)$ are given by
\[\mu = \mathrm{IB}(\sigma,\rho),\qquad \nu = \tBeta(\rho+\sigma,\tau),\qquad \tilde{\mu} = \mathrm{IB}(\sigma+\tau,\rho),\qquad \tilde{\nu} = \tBeta(\sigma,\tau).\]
(ii) The only non-trivial solutions of the detailed balance condition for $F_{(1,1)}:(0,\infty)\times(1,\infty)\rightarrow(0,\infty)^2$ are given by
\[\mu = \tBeta(\rho+\sigma,\tau),\qquad \nu =\mathrm{IB}(\sigma,\rho),\qquad \tilde{\mu} = \tBeta(\rho,\sigma+\tau),\qquad \tilde{\nu} =\tBeta(\tau,\sigma).\]
(d)  The only non-trivial solutions of the detailed balance condition for $F_{(1,-1)}:(0,1)\times(1,\infty)\rightarrow(0,1)\times(1,\infty)$ are given by
\[\mu = \mathrm{Be}(\rho+\sigma,\tau),\qquad \nu =\mathrm{IB}(\rho,\sigma),\qquad \tilde{\mu} = \mathrm{Be}(\sigma,\tau),\qquad \tilde{\nu} = \mathrm{IB}(\rho,\sigma+\tau).\]
\end{thm}

To complete the discussion in the positive-temperature case we note that, by applying this result together with Proposition \ref{twothree}, one thus obtains a characterisation of the stationary measures of the maps $R_{(0,1)}$, $R_{(1,0)}$, $R_{(-1,1)}$ (or $R_{(1,1)}$) and $R_{(1,-1)}$, as desired. (We recall that this was already done in \cite{CN}, and point to Figure \ref{fig1} for a summary.)

\begin{rem}
We note that the names of the models are taken from what is $\tilde{\mu}$ in our notation, that is, the distribution of the random variables $X_{n,m}$. To help relate our results to the literature, we recall parameterizations of $\tilde{\mu}\times\mu\times\nu$ that are often used (see, for example, \cite{CN}, \cite{Thiery}).
\begin{itemize}
\item For the inverse-gamma model, $R_{(0,1)}$, with $\gamma > \lambda >0$, $\beta >0$:
\[\tilde{\mu}=\mathrm{IG}(\gamma,\beta),\qquad \mu=\mathrm{IG}(\gamma-\lambda, \beta),\qquad\nu=\mathrm{IG}(\lambda, \beta).\]
\item For the gamma model, $R_{(1,0)}$, with $\gamma,\lambda,\beta>0$:
\[\tilde{\mu}=\mathrm{Gam}(\gamma,\beta), \qquad \mu=\mathrm{Gam}(\gamma+\lambda,\beta),\qquad \nu=\mathrm{IB}(\la,\gamma).\]
\item For the inverse-beta model, $R_{(-1,1)}$, with $\gamma> \lambda >0$, $\beta >0$:
\[\tilde{\mu}=\mathrm{IB}(\gamma,\beta),\qquad \mu=\mathrm{IB}(\gamma-\lambda,\beta),\qquad \nu= \mathrm{Be}'(\beta+\gamma-\lambda,\la),\]
or $R_{(1,1)}$, with $\gamma> \lambda >0$, $\beta >0$:
\[\tilde{\mu}=\mathrm{Be}'(\beta,\gamma),\qquad \mu=\mathrm{Be}'(\beta+\lambda,\gamma-\lambda),\qquad \nu =\mathrm{IB}(\la,\beta).\]
\item For the beta model, $R_{(1,-1)}$, with $\gamma,\lambda,\beta>0$:
\[\tilde{\mu}=\mathrm{ Be}(\gamma,\beta),\qquad \mu=\mathrm{ Be}(\gamma+\lambda,\beta),\qquad \nu=\mathrm{IB}(\la,\gamma).\]
\end{itemize}
\end{rem}

\begin{rem} From the point of view of polymer models, as is explained in the discussion surrounding \cite[Theorem 1.2]{CN}, it is natural to consider $F_{(\a,\b)}$ with $\max\{\a,\b\}>0$. As mathematical objects, one might also ask what transpires when this restriction on parameters is dropped. In this direction, we note that
\[F_{(-\a,-\b)}=\left(I_d\times (-I_d)\right)\circ F_{(\a,\b)} \circ\left(I_d\times (-I_d)\right),\]
which further implies
\[R_{(-\a,-\b)}=\left(I_d\times (-I_d)\right)\circ R_{(\a,\b)} \circ\left(I_d\times I_d\times (-I_d)\right).\]
Thus Theorem \ref{t42} gives stationary measures with $\nu$ supported on negative values for $R_{(-1,-1)}$, $R_{(-1,0)}$ and $R_{(0,-1)}$, for example. As defined in this article, zero-temperature limits are not defined for functions supported on negative domains, and so this observation does not extend to the zero-temperature case.
\end{rem}

\subsection{Zero-temperature limits}\label{s42}

In the zero-temperature setting, the analogue of Proposition \ref{gb} is the following. We highlight that part (b) does not provide a complete characterisation of detailed balance solutions for $F_{\AL,\AL}$; see Section \ref{sec6} for a conjecture on this case.

\begin{prop}\label{gbz}
(a) The only non-trivial solutions of the detailed balance condition for $F_{\E,\AL}:\mathbb{R}^2\rightarrow\mathbb{R}^2$ are given by
\[\mu = \mathrm{sExp}(\rho,\tau),\qquad \nu = \mathrm{sExp}(\sigma,\tau),\qquad\tilde{\mu} = \mathrm{sExp}(\rho+\sigma,\tau),\qquad\tilde{\nu} = \AL(\rho,\sigma),\]
where $\rho,\sigma>0$ and $\tau\in \R$, and
\[\mu= \mathrm{ssGeo}(p,M,m),\quad \nu = \mathrm{ssGeo}(q,M,m),\quad \tilde{\mu}=\mathrm{ssGeo}(pq,M,m), \quad \tilde{\nu} =  \mathrm{sdAL}(p,q,m),\]
where $p,q\in(0,1)$, $M \in \Z$ and $m\in(0,\infty)$.\\
(b) Non-trivial solutions of the detailed balance condition for $F_{\AL,\AL}:\mathbb{R}^2\rightarrow\mathbb{R}^2$ are given by
\[\mu= \mathrm{AL}(\rho,\sigma),\qquad \nu=  \mathrm{AL}(\rho+\sigma,\tau),\qquad \tilde{\mu}= \mathrm{AL}(\tau,\sigma),\qquad \tilde{\nu}= \mathrm{AL}(\sigma+\tau,\rho),\]
where $\rho,\sigma,\tau > 0$, and
\[\mu=\mathrm{sdAL}(p,q,m),\qquad\nu=\mathrm{sdAL}(pq,r,m),\qquad\tilde{\mu}=\mathrm{sdAL}(r,q,m),\qquad \tilde{\nu}= \mathrm{sdAL}(qr,p,m),\]
where $0< p,q,r <1$ and $m\in(0,\infty)$.
\end{prop}

\begin{proof} The forms of $\mu$ and $\nu$ in part (a) were established in \cite{Cr}. Deducing the forms of $\tilde{\mu}$ and $\tilde{\nu}$ is straightforward. For part (b), checking the claimed implication is also straightforward, and so we only sketch the proof of the continuous version of the result. Suppose $(U,V)\sim \tilde{\mu}\times\tilde{\nu}$ for $\tilde{\mu}= \mathrm{AL}(\tau,\sigma)$ and $\tilde{\nu}= \mathrm{AL}(\sigma+\tau,\rho)$. This joint distribution has density given by
\[\frac{1}{Z}\left(e^{-\tau u}\mathbf{1}_{(0,\infty)}(u)+e^{\sigma u}\mathbf{1}_{(-\infty,0)}(u)\right)\left(e^{-(\sigma+\tau) v}\mathbf{1}_{(0,\infty)}(v)+e^{\rho v}\mathbf{1}_{(-\infty,0)}(v)\right),\]
where $Z$ is a normalizing constant. Now, since the Jacobian of $F_{\AL,\AL}$ has absolute value 1 Lebesgue-almost-everywhere, the density of $(X,Y)=F_{\AL,\AL}^{-1}(U,V)$ is given by
\begin{align*}
&\frac{1}{Z}\left(e^{-\tau ((0\wedge x)-y)}\mathbf{1}_{(0,\infty)}((0\wedge x)-y)+e^{\sigma((0\wedge x)-y)}\mathbf{1}_{(-\infty,0)}((0\wedge x)-y)\right)\\
&\times \left(e^{-(\sigma+\tau)((0\wedge x\wedge y)-x-y)}\mathbf{1}_{(0,\infty)}((0\wedge x\wedge y)-x-y)+e^{\rho ((0\wedge x\wedge y)-x-y)}\mathbf{1}_{(-\infty,0)}((0\wedge x\wedge y)-x-y)\right).
\end{align*}
By breaking $\mathbb{R}^2$ into five regions (i.e.\ $x,y>0$; $x>0>y$; $x<0<y$; $x<y<0$; $y<x<0$), it is an elementary exercise to check that the above expression is equal to
\[\frac{1}{Z}\left(e^{-\rho x}\mathbf{1}_{(0,\infty)}(x)+e^{\sigma x}\mathbf{1}_{(-\infty,0)}(x)\right)\left(e^{-(\rho+\sigma) y}\mathbf{1}_{(0,\infty)}(y)+e^{\tau y}\mathbf{1}_{(-\infty,0)}(y)\right),\]
which confirms that if $\mu= \mathrm{AL}(\rho,\sigma)$ and $\nu=  \mathrm{AL}(\rho+\sigma,\tau)$, then $(\mu,\nu,\tilde\mu,\tilde\nu)$ is a solution of the detailed balance condition for $F_{\AL,\AL}$, as desired.
\end{proof}

To transfer this proposition to the maps $R^\z_{(0,1)}$, $R^\z_{(1,0)}$, $R^\z_{(1,1)}$ and $\tilde{R}^\z_{(1,-1)}$, we will apply the following conclusion concerning changes of variables. Given Proposition \ref{twothree}, the result is obvious in the case where the change of variables is bijective. The purpose of including it is to clarify the implication that holds when this is not the case.

\begin{lem}\label{lemlem}
Let $F: \R^2 \to \R^2$ be a bijection and suppose there are functions $f_1 : \R \to I_1$, $f_2 : \R \to I_2$, $\tilde{f}_1 : \R \to J_1$ such that
\[R\circ(\tilde{f}_1\times f_1\times f_2)= (f_1 \times f_2) \circ \overline{F}^{(2,3)}\]
on $\R^3$. If $(\mu', \nu',\tilde{\mu}' ,\tilde{\nu}')$ is a quadruplet of probability measures on $\mathbb{R}$ that satisfy the detailed balance condition for $F$, then \eqref{Rinv} holds with
\[\mu:=f_1(\mu'),\qquad \nu:=f_2(\nu'),\qquad \tilde{\mu}:=\tilde{f}_1(\tilde{\mu}').\]
Moreover, if $f_1,f_2$ and $\tilde{f}_1$ are all bijections, then the reverse implication is also true with $\tilde{\nu}':=F^{(2)}(f_1^{-1}(\mu)\times f_2^{-1}(\nu))$.
\end{lem}

\begin{proof}
Suppose $(X',U',V')\sim  \tilde{\mu}'\times \mu'\times \nu'$, and set $X:=\tilde{f}_1(X')$, $U:= f_1(U')$, $V:=f_2(V')$, so that $(X,U,V)\sim\tilde{\mu}\times \mu\times \nu$. We then have that if $(U^{*},V^{*}):=R(X,U,V)$, then
\[(U^{*},V^{*})=(f_1\times f_2)\circ\overline{F}^{(2,3)}(X',U',V').\]
Since by Proposition \ref{twothree}, $\overline{F}^{(2,3)}(X',U',V')\sim \mu'\times\nu'$, we obtain that $(U^{*},V^{*})\sim\mu\times\nu$, as required to establish the first claim. Applying Proposition \ref{twothree} again, the second claim is obvious.
\end{proof}

From Propositions \ref{p34} and \ref{gbz}, Lemma \ref{lemlem} and the fact that $-\AL(\rho,\sigma)=\AL(\sigma,\rho)$, one may read off the following conclusion concerning the stationary measures of our zero-temperature models. Note that only for $R_{(0,1)}^\z$ do we have a complete characterisation of stationary measures, since it is only in this case that we have both a complete characterization of the basic bijection, and a bijective change of scale linking the basic bijection to the original dynamics. Whilst the solutions presented in parts (a)-(c) have essentially been presented previously (albeit without the shifted versions of the exponential/geometric distributions that appear in our result, see Remark \ref{trem}), part (d) is apparently new.

\begin{thm}\label{t45}
(a) The map $R^{\z}_{(0,1)}:\R^3 \to \R^2$ satisfies \eqref{Rinv} if and only if
\[\mu = -\mathrm{sExp}(\rho,\tau),\qquad \nu = -\mathrm{sExp}(\sigma,\tau),\qquad\tilde{\mu} = -\mathrm{sExp}(\rho+\sigma,\tau),\]
where $\rho,\sigma>0$ and $\tau\in \R$, or
\[\mu= -\mathrm{ssGeo}(p,M,m),\qquad \nu = -\mathrm{ssGeo}(q,M,m),\qquad \tilde{\mu}=-\mathrm{ssGeo}(pq,M,m),\]
where $p,q\in(0,1)$, $M \in \Z$ and $m\in(0,\infty)$.\\
(b) The map $R^\z_{(1,0)} : \R \times \R \times (-\infty,0] \to  \R \times (-\infty,0]$ satisfies \eqref{Rinv} if
\[{\mu} = \mathrm{sExp}(\rho+\sigma,\tau),\qquad{\nu}=\AL(\sigma,\rho)\wedge 0,\qquad \tilde{\mu} = \mathrm{sExp}(\sigma,\tau),\]
where $\rho,\sigma>0$ and $\tau\in \R$, or
\[{\mu}=\mathrm{ssGeo}(pq,M,m), \qquad {\nu} =  \mathrm{sdAL}(q,p,m)\wedge 0,\qquad \tilde{\mu}= \mathrm{ssGeo}(q,M,m),\]
where $p,q\in(0,1)$, $M \in \Z$ and $m\in(0,\infty)$.\\
(c) The map $R^{\z}_{(1,1)} : \R \times \R \times (-\infty,0]  \to \R \times (-\infty,0] $ satisfies \eqref{Rinv} if
\[\mu=  \mathrm{AL}(\rho+\sigma,\tau),\qquad \nu= \mathrm{AL}(\rho,\sigma)\wedge 0,\qquad\tilde{\mu}= \mathrm{AL}(\rho,\sigma+\tau),\]
where $\rho,\sigma,\tau > 0$, or
\[\mu=\mathrm{sdAL}(pq,r,m),\qquad \nu=\mathrm{sdAL}(p,q,m)\wedge 0,\qquad \tilde{\mu}= \mathrm{sdAL}(p,qr,m),\]
where $0< p,q,r <1$ and $m\in(0,\infty)$.\\
(d) The map $\tilde{R}^{\z}_{(1,-1)}: \R \times [0,\infty) \times (-\infty,0] \to [0,\infty) \times (-\infty,0]$ satisfies \eqref{Rinv} if
\[\mu= \mathrm{AL}(\rho+\sigma,\tau)\vee 0,
\qquad \nu=  \mathrm{AL}(\sigma,\rho)\wedge0,\qquad \tilde{\mu}= \mathrm{AL}(\tau,\sigma),\]
where $\rho,\sigma,\tau > 0$, or
\[\mu=\mathrm{sdAL}(pq,r,m)\vee 0,\qquad\nu=\mathrm{sdAL}(q,p,m)\wedge0,\qquad\tilde{\mu}=\mathrm{sdAL}(r,q,m),\]
where $0< p,q,r <1$ and $m\in(0,\infty)$.
\end{thm}

\begin{rem} Applying \eqref{rrr2} and Proposition \ref{gbz}, we find that the only non-trivial solutions of the detailed balance condition for ${F}^\z_{(0,1)}$ are given by
\[\mu = -\mathrm{sExp}(\rho,\tau),\qquad -\nu = \mathrm{sExp}(\sigma,\tau),\qquad\tilde{\mu} = -\mathrm{sExp}(\rho+\sigma,\tau),\qquad\tilde{\nu} = \AL(\rho,\sigma),\]
where $\rho,\sigma>0$ and $\tau\in \R$, and
\[\mu= -\mathrm{ssGeo}(p,M,m),\quad \nu = -\mathrm{ssGeo}(q,M,m),\quad \tilde{\mu}=-\mathrm{ssGeo}(pq,M,m), \quad \tilde{\nu} =  \mathrm{sdAL}(p,q,m),\]
where $p,q\in(0,1)$, $M \in \Z$ and $m\in(0,\infty)$.
\end{rem}

\begin{rem}\label{trem} (a) For $R^{\z}_{(0,1)}$, the following stationary solutions are presented in the literature: for $\gamma>\la>0$,
\[\tilde{\mu}= -\mathrm{Exp}(\gamma),\qquad\mu= -\mathrm{Exp}(\gamma-\lambda), \qquad\nu= -\mathrm{Exp}(\lambda),\]
see \cite[(V.12)]{Thiery} (but note there is a typo in that the distributions of $U$ and $V$ are exchanged).\\
(b) For  $R^\z_{(1,0)}$, \cite[(V.13)]{Thiery} gives the stationary solutions: for $\beta, \la >0$,
\[\tilde{\mu}=\mathrm{Exp}(\beta),\qquad \mu=\mathrm{Exp}(\beta+\lambda),\qquad\nu=\AL(\beta,\lambda)\wedge 0.\]
(c) For $R^{\z}_{(1,1)}$, \cite[(V.11)]{Thiery} gives the stationary solutions: for $\gamma>\la>0$, $\beta >0$,
\[\tilde{\mu}= \AL(\beta,\gamma), \qquad \mu = \AL(\beta+\lambda,\gamma-\lambda),\qquad \nu=\AL(\beta,\lambda)\wedge0.\]
Discrete versions of these solutions are also presented in \cite{Thiery}.
\end{rem}

To put our approach into greater context, let us note that in identifying the stationary measures of a zero-temperature model, a common starting point is to take appropriate zero-temperature limits of the stationary measures of the corresponding positive-temperature model. In general, such an approach only identifies continuous stationary measures (cf. Proposition \ref{ztlmapmeas} below), and misses the discrete ones, which then have to be found in an ad hoc fashion. For instance, as is commented in \cite{Thiery}, the definition of the Bernoulli-geometric polymer was found by `trial and error'. An advantage of our first step of relating each of the zero-temperature polymer models to one of two basic bijections, albeit non-bijectively, is that it readily allows us to read off both discrete and continuous stationary measures. For example, the expression for $\nu=\AL(\sigma,\rho)\wedge 0$ (and its discrete version) in Theorem \ref{t45}(b), which is new, readily follows from Propositions \ref{p34} and \ref{gbz}. Indeed, putting the latter results together gives that the relevant $\nu$ is, in the continuous case, given as:
\[\nu=I^\z \circ (Q^{-1})^\z \circ I^\z \left(\AL(\rho,\sigma)\right)=I^\z \circ (Q^{-1})^\z \left(\AL(\sigma,\rho)\right)=\AL(\sigma,\rho)\wedge0,\]
and, in particular, we clearly see that the atom in the distribution $\nu$ at 0 arises from the cut-off of the asymmetric Laplace distribution by the non-bijective map $(Q^{-1})^\z$. As we discuss in the next section, in both the positive- and zero-temperature cases, the picture in terms of basic bijections also helps to clarify the connections between the stationary measures of different polymer models.

To finish this section, we further note that one can obtain solvable polymer models as limits of stochastic vertex models, such as the six vertex model and higher-spin vertex models, for which stationary versions have been studied; for background, see \cite{AA, C2,C3, IMS}. Moreover, the beta polymer arises as a limit of the $q$-Hahn Boson process, a stationary version of which is discussed in \cite{C1}. Thus, another potential route for deriving the stationary measures of polymer models is via their connections with these other models. As far as we are aware, such limits have not yet been explored systematically in the literature.

\section{Relationships between models}\label{sec5}

In this section, we explore relationships between the basic bijections $F_{\tBeta,\tBeta}$, $F_{\Gam,\tBeta}$, $F_{\AL,\AL}$ and $F_{\E,\AL}$, and also the associated detailed balance solutions. We break the discussion into three parts. Firstly, in the positive-temperature setting, we start by describing how both $F_{\Gam,\tBeta}$ and $F_{\Gam,\tBeta}^{-1}$ can be deduced as suitable limits of $F_{\tBeta,\tBeta}$, and also derive relationships between the various detailed balance solutions; see Subsection \ref{ptsec}. Secondly, in Subsection \ref{ztlsec}, we present both bijection and detailed balance solution versions of zero-temperature limits for $F_{\tBeta,\tBeta}$, $F_{\Gam,\tBeta}$ and  $F_{\Gam,\tBeta}^{-1}$. Thirdly, in Subsection \ref{ztsec}, we set out zero-temperature versions of the results of Subsection \ref{ptsec}, which include links between $F_{\AL,\AL}$ and both $F_{\E,\AL}$ and $F_{\E,\AL}^{-1}$. These results are summarised in Figure \ref{fig2}. We also discuss how, by making suitable changes of variables, corresponding links for the random polymer models and their stationary measures can be obtained in certain cases, see Subsection \ref{ptrem}, as well as Remarks \ref{ptztrem} and \ref{ztrem}. Note that the links we present essentially recover those illustrated in \cite[Figure 1]{Thiery} (see also \cite[Figure 2]{TD}), and also include a novel one between the zero temperature limits of the stationary beta and gamma random polymers (see Remark \ref{ztrem} in particular).

\begin{figure}
\centerline{\xymatrix@R-9pt{\framebox(50,30){\parbox{50pt}{\centering{$F_{\Gam,\tBeta}$}}}\ar@{.>}[d]&
\framebox(50,30){\parbox{50pt}{\centering{$F_{\tBeta,\tBeta}$}}}\ar@<-12pt>[l]\ar@<12pt>[r]\ar@{.>}[d]
&\framebox(50,30){\parbox{50pt}{\centering{$F_{\Gam,\tBeta}^{-1}$}}}\ar@{.>}[d]\\
\framebox(50,30){\parbox{50pt}{\centering{$F_{\E,\AL}$}}}&
\framebox(50,30){\parbox{50pt}{\centering{$F_{\AL,\AL}$}}}\ar@<-12pt>@{-->}[l]\ar@<12pt>@{-->}[r]
&\framebox(50,30){\parbox{50pt}{\centering{$F_{\E,\AL}^{-1}$}}}}}
\caption{Relationships between basic bijections (and corresponding detailed balance solutions). The solid arrows between positive-temperature models are established in Subsection \ref{ptsec}. The dotted arrows are zero-temperature limits, and are made precise in Subsection \ref{ztlsec}. The dashed arrows between zero-temperature models are explained in Subsection \ref{ztsec}.}\label{fig2}
\end{figure}
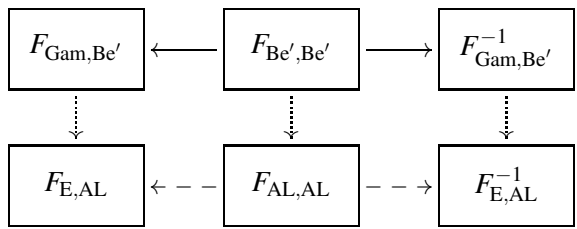

\subsection{Relationships between positive-temperature models}\label{ptsec}

For the positive-temperature versions of our basic bijections, we have the following result.

\begin{prop}\label{ptmap} Uniformly on compact subsets of $(0,\infty)^2$, it holds that
\[\lim_{\delta\rightarrow0}\pi\circ\left(I\times \delta^{-1} I_d\right)\circ F_{\tBeta,\tBeta}\circ \left(\delta^{-1}I\times \delta^{-1}I\right)=F_{\Gam,\tBeta},\]
and also
\[\lim_{\delta\rightarrow0}\left(\delta^{-1}I\times \delta^{-1}I\right)\circ F_{\tBeta,\tBeta}\circ\left(I\times \delta I_d\right)\circ\pi=F_{\Gam,\tBeta}^{-1}\]
\end{prop}
\begin{proof} Since
\[\pi\circ\left(I\times \delta^{-1} I_d\right)\circ F_{\tBeta,\tBeta}\circ \left(\delta^{-1}I\times \delta^{-1}I\right)(x,y)=\left(x+y+\delta xy,\frac{x}{y+\delta xy}\right),\]
the first statement is clear. Checking the second statement is similarly straightforward.
\end{proof}

Using the changes of scale in the above proposition, we have a corresponding connection between the detailed balance solutions of the bijections in question. The limits should be interpreted as those given by convergence in distribution.

\begin{prop}\label{ptmeas}
With $(\rho_\delta,\sigma_\delta,\tau_\delta)=(\delta^{-1}\tau,\rho,\sigma)$ for some $\rho,\sigma,\tau\in(0,\infty)$, we have
\[\lim_{\delta\rightarrow0}\left(\delta^{-1}I\times \delta^{-1}I\right)\left(\tBeta(\rho_\delta,\sigma_\delta) \times \tBeta(\rho_\delta+\sigma_\delta,\tau_\delta)\right) =\Gam(\rho,\tau)\times \Gam(\sigma,\tau),\]
and also
\[\lim_{\delta\rightarrow0}\pi\circ\left(I\times \delta^{-1} I_d\right)\left( \tBeta(\tau_\delta,\sigma_\delta)\times\tBeta(\sigma_\delta+\tau_\delta,\rho_\delta)\right)=\Gam(\rho+\sigma,\tau)\times \tBeta(\rho,\sigma).\]
\end{prop}
\begin{proof} Suppose $X_\delta\sim \tBeta(\rho_\delta,\sigma_\delta)$, and set $Y_\delta:=\delta^{-1}X^{-1}_\delta$. An elementary change of variables yields that the density of $Y_\delta$ is given by
\[\frac{\delta\left(\delta y\right)^{\sigma_\delta-1}\left(1+\delta y\right)^{-(\rho_\delta+\sigma_\delta)}}{B(\rho_\delta,\sigma_\delta)},\]
where $B(\rho,\sigma)$ is the beta function. Using that $\delta^\rho\Gamma(\frac{\tau}{\delta}+\rho)/\Gamma(\frac{\tau}{\delta})\rightarrow \tau^\rho$ (where $\Gamma(x)$ is the gamma function), one readily obtains that the above formula converges (uniformly on compacts) to the density of a $\Gam(\rho,\tau)$ distribution. Dealing with the other marginals is similar.
\end{proof}

To complete our discussion of the links between the positive-temperature bijections $F_{\tBeta,\tBeta}$ and $F_{\Gam,\tBeta}$, we explain how the preceding two statements can be used to connect their detailed balance solutions in a rigourous way. In particular, from Proposition \ref{ptmap}, we have that
\begin{equation}\label{l1}
\phi_\delta\circ F_{\tBeta,\tBeta}\circ \psi_\delta\rightarrow  F_{\Gam,\tBeta}
\end{equation}
as $\delta\rightarrow 0$, where $\phi_\delta:=\pi\circ(I\times \delta^{-1} I_d)$ and $\psi_\delta:=(\delta^{-1}I\times \delta^{-1}I)$ are bijective changes of scale on $(0,\infty)^2$. Moreover, Proposition \ref{ptmeas} tells us that if
\begin{equation}\label{m1}
\mu_\delta := \tBeta(\rho_\delta,\sigma_\delta),\quad \nu_\delta := \tBeta(\rho_\delta+\sigma_\delta,\tau_\delta),\quad \tilde{\mu}_\delta := \tBeta(\tau_\delta,\sigma_\delta),\quad \tilde{\nu}_\delta := \tBeta(\sigma_\delta+\tau_\delta,\rho_\delta),
\end{equation}
then
\begin{equation}\label{l2}
\psi^{-1}_\delta\left(\mu_{\delta}\times \nu_\delta\right)\rightarrow \mu\times\nu,\qquad
\phi_\delta\left(\tilde{\mu}_{\delta}\times \tilde{\nu}_\delta\right)\rightarrow \tilde{\mu}\times\tilde{\nu},
\end{equation}
where
\begin{equation}\label{m2}
\mu := \Gam(\rho,\tau),\qquad \nu := \Gam(\sigma,\tau),\qquad\tilde{\mu}:= \Gam(\rho+\sigma,\tau),\qquad \tilde{\nu}:= \tBeta(\rho,\sigma).
\end{equation}
Since from Proposition \ref{gb}(b) we know that, for each $\delta>0$, the measures at \eqref{m1} solve the detailed balance condition for $F_{\tBeta,\tBeta}$, one readily obtains from \eqref{l1} and \eqref{l2} that the measures at \eqref{m2} solve the detailed balance condition for $F_{\Gam,\tBeta}$, thus giving an alternative proof of one of the directions of implication of Proposition \ref{gb}(a). One may similarly use the second statement of Proposition \ref{ptmap} to connect to the detailed balance solutions for $F_{\Gam,\tBeta}^{-1}$.

\subsubsection{Relations between positive-temperature random polymer models}\label{ptrem}

Applying the above arguments, one can further obtain a number of connections between random polymer models. Specifically, it is possible to explain how the maps describing the gamma/inverse-gamma random polymer models, and their stationary solutions, can be deduced from the maps and measures arising in the inverse-beta model. The convergence of maps in particular illuminates previous results connecting the partition functions of the models (which also follow from our conclusions). On the other hand, although it is also possible to connect the beta model to the gamma model in a way that yields convergence of the partition function, the relation we describe between the beta and inverse-gamma cases is such that we can not link the partition functions of the two models.

We start by combining Propositions \ref{p31} and \ref{ptmap} to deduce the following limiting results for the underlying bijections $F_{(\alpha,\beta)}$. For parts (a)-(b), we use that $\delta Q I (\delta^{-1}x)\rightarrow x$ as $\delta\rightarrow 0$ (uniformly on compacts) to simplify some of the changes of variables to linear ones, and also choose between $F_{(-1,1)}$ and $F_{(1,1)}$ in order to avoid the appearance of the coordinate transposition $\pi$. Part (c) involves the non-linear transformation $I\circ J\circ I$ (where we recall that $J:=Q^{-1}\circ I\circ Q$) and $\pi$, but somewhat remarkably (as we discuss in Remark \ref{ptptrem} below), these operations fit perfectly with the invariant measures, so that we can still obtain convergence of the partition functions. Finally, in part (d), we require a second order term to obtain the relevant limit (clearly the first order term in the change of variables is degenerate); this difference seems to be a fundamental issue when it comes to linking the beta and inverse-gamma models.

\begin{prop}\label{p53} The following limits hold uniformly on compact subsets of the domain of the limiting bijection.\\
(a) (Inverse-beta to inverse-gamma.)
\[\lim_{\delta\rightarrow0} \left(\delta I_d\times I_d\right)\circ F_{(-1,1)}\circ\left(\delta^{-1} I_d\times \delta^{-1}I_d\right)=F_{(0,1)}.\]
(b) (Inverse-beta to gamma.)
\[\lim_{\delta\rightarrow0} \left(\delta^{-1} I_d\times \delta I_d\right)\circ F_{(1,1)}\circ\left(\delta I_d\times I_d\right)=F_{(1,0)}.\]
(c) (Beta to gamma.)
\[\lim_{\delta\rightarrow0}\pi\circ\left(\delta I\times \delta^{-1} I\right)\circ F_{(1,-1)}\circ\left(\delta I_d\times (I\circ J\circ I)\right)=F_{(1,0)}.\]
(d) (Beta to inverse-gamma.)
\[\lim_{\delta\rightarrow0}\pi\circ\left((I\circ Q)\times( \delta I\circ Q\circ I)\right)\circ F_{(1,-1)}\circ\left((\mathbf{1}-\delta I)\times (\mathbf{1}+\delta I)\right)\circ \pi=F_{(0,1)},\]
where $\mathbf{1}$ is the function taking the value $1$ everywhere.
\end{prop}

As a simple corollary of the first two cases of the preceding result, one is able to deduce the corresponding limit for the maps $R_{(\alpha,\beta)}$. The presence of the transposition operator $\pi$ means that it is not possible to read off the analogous results for cases (c) and (d).

\begin{cor}\label{c54} The following limits hold uniformly on compact subsets of the domain of the limiting map.\\
(a) (Inverse-beta to inverse-gamma.)
\[\lim_{\delta\rightarrow0} \left(\delta I_d\times \delta I_d\right)\circ R_{(-1,1)}\circ\left(\delta^{-1} I_d\times \delta^{-1} I_d\times \delta^{-1}I_d\right)=R_{(0,1)}.\]
(b) (Inverse-beta to gamma.)
\[\lim_{\delta\rightarrow0} \left(\delta^{-1} I_d\times I_d\right)\circ R_{(1,1)}\circ\left(\delta I_d\times \delta I_d\times I_d\right)=R_{(1,0)}.\]
\end{cor}

With respect to the changes of variables identified in Proposition \ref{p53}, we have the following relationships between invariant measures. The proof is similar to that of Proposition \ref{ptmeas} and is omitted. (The limits of measures are weak limits as $\delta\rightarrow 0$.)

\begin{prop}\label{p55}
(a) (Inverse-beta to inverse-gamma.) Suppose $(\mu_\delta,\nu_\delta,\tilde{\mu}_\delta,\tilde{\nu}_\delta)$ are defined as in Theorem \ref{t42}(c)(i) with parameters $(\rho,\sigma,\tau)$ given by $(\delta^{-1}\tau,\rho,\sigma)$. It then holds that
\[\left(\delta I_d\times \delta I_d\right)\left(\mu_\delta\times\nu_\delta\right)\rightarrow \mu\times\nu,\qquad \left(\delta I_d\times I_d\right)\left(\tilde{\mu}_\delta\times\tilde{\nu}_\delta\right)\rightarrow \tilde{\mu}\times \tilde{\nu},\]
where $(\mu,\nu,\tilde{\mu},\tilde{\nu})$ are defined as in Theorem \ref{t42}(a).\\
(b) (Inverse-beta to gamma.) Suppose $(\mu_\delta,\nu_\delta,\tilde{\mu}_\delta,\tilde{\nu}_\delta)$ are defined as in Theorem \ref{t42}(c)(ii) with parameters $(\rho,\sigma,\tau)$ given by $(\sigma,\rho,\delta^{-1}\tau)$. It then holds that
\[\left(\delta^{-1} I_d\times I_d\right)\left(\mu_\delta\times\nu_\delta\right)\rightarrow \mu\times\nu,\qquad \left(\delta^{-1} I_d\times\delta I_d\right)\left(\tilde{\mu}_\delta\times\tilde{\nu}_\delta\right)\rightarrow \tilde{\mu}\times \tilde{\nu},\]
where $(\mu,\nu,\tilde{\mu},\tilde{\nu})$ are defined as in Theorem \ref{t42}(b).\\
(c) (Beta to gamma.) Suppose $(\mu_\delta,\nu_\delta,\tilde{\mu}_\delta,\tilde{\nu}_\delta)$ are defined as in Theorem \ref{t42}(d) with parameters $(\rho,\sigma,\tau)$ given by $(\sigma,\rho,\delta^{-1}\tau)$. It then holds that
\[\left(\delta^{-1} I_d\times (I\circ J\circ I)\right)\left(\mu_\delta\times\nu_\delta\right)\rightarrow \mu\times\nu,\qquad\pi\circ\left(\delta I\times \delta^{-1} I\right)\left(\tilde{\mu}_\delta\times\tilde{\nu}_\delta\right)\rightarrow \tilde{\mu}\times \tilde{\nu},\]
where $(\mu,\nu,\tilde{\mu},\tilde{\nu})$ are defined as in Theorem \ref{t42}(b).\\
(d) (Beta to inverse-gamma.) Suppose $(\mu_\delta,\nu_\delta,\tilde{\mu}_\delta,\tilde{\nu}_\delta)$ are defined as in Theorem \ref{t42}(d) with parameters $(\rho,\sigma,\tau)$ given by $(\delta^{-1}\tau,\rho,\sigma)$. It then holds that
\[\pi\circ\left((\mathbf{1}-\delta I)^{-1}\times (\mathbf{1}+\delta I)^{-1}\right)\left(\mu_\delta\times\nu_\delta\right)\rightarrow \mu\times\nu,\hspace{80pt}\]
\[\hspace{80pt}\pi\circ\left((I\circ Q)\times( \delta I\circ Q\circ I)\right)\left(\tilde{\mu}_\delta\times\tilde{\nu}_\delta\right)\rightarrow \tilde{\mu}\times \tilde{\nu},\]
where $(\mu,\nu,\tilde{\mu},\tilde{\nu})$ are defined as in Theorem \ref{t42}(a).
\end{prop}

\begin{rem}\label{ptptrem}
We now discuss case-by-case the implications of Proposition \ref{p55} for the partition functions of the random polymer models.\\
(a) For each $\delta$, let $\tilde{\mu}_\delta$, $\mu_\delta$ and $\nu_\delta$ be defined as in Proposition \ref{p55}(a). Consider the stationary inverse-beta random polymer corresponding to $R_{(-1,1)}$ with $(X^\delta_{n,m})_{n,m \in \mathbb{N}}$ being i.i.d.\ $\tilde{\mu}_\delta$ random variables, and boundary condition being given by independent random variables $(U_{n,0})_{ n \in \N}$ and $(V_{0,n})_{n \in \N}$ with marginals $\mu_\delta$ and $\nu_\delta$, respectively, and $Z^\delta_{0,0}=1$. It then immediately follows from Corollary \ref{c54} and Proposition \ref{p55} that the associated partition function $(Z^\delta_{n,m})_{n,m \in \mathbb{Z}_+}$ satisfies
\[\left(\delta^{n+m}Z^\delta_{n,m}\right)_{n,m \in \mathbb{Z}_+}\rightarrow \left(Z_{n,m}\right)_{n,m \in \mathbb{Z}_+}\]
in distribution as $\delta\rightarrow 0$, where $(Z_{n,m})_{n,m \in \mathbb{Z}_+}$ is the partition function of the stationary inverse-gamma model corresponding to $R_{(0,1)}$, as defined from the limiting measures $\tilde{\mu}$, $\mu$ and $\nu$. We note that the stationary $R_{(-1,1)}$ model with $\tilde{\mu}_\delta:=\mathrm{IB}(\rho+\sigma,\delta^{-1}\tau)$ is simply a reflection in the line $x=y$ of the stationary inverse-beta random polymer corresponding to $R_{(1,1)}$ with
\[\tilde{\mu}_\delta:=(Q\circ I)\left(\mathrm{IB}(\rho+\sigma,\delta^{-1}\tau)\right)=\mathrm{Be}'(\delta^{-1}\tau,\rho+\sigma)=\frac{1}{\mathrm{Be}(\rho+\sigma,\delta^{-1}\tau)}-1.\]
(Cf.\ \cite[Lemma 4.1(f)]{CN}.) Hence the above partition function limit, which stems from the divergence of second parameter of the inverse-beta distribution, matches that described in \cite[Section IV.D.1]{Thiery} (see also \cite{TD}).\\
(b) A similar argument applies in case (b), where $(Z^\delta_{n,m})_{n,m \in \mathbb{Z}_+}$ is the partition function of the stationary inverse-beta random polymer corresponding to $R_{(1,1)}$ and $(Z_{n,m})_{n,m \in \mathbb{Z}_+}$ is the partition function of the stationary inverse-gamma model corresponding to $R_{(1,0)}$. However, since in this case we only have scaling in one-direction, the appropriate form of the result is as follows:
\[\left(\delta^{-n}Z^\delta_{n,m}\right)_{n,m \in \mathbb{Z}_+}\rightarrow \left(Z_{n,m}\right)_{n,m \in \mathbb{Z}_+}\]
in distribution as $\delta\rightarrow 0$. Now,
\[\tilde{\mu}_\delta:=\mathrm{Be}'(\sigma,\rho+\delta^{-1}\tau)=(Q\circ I)\left(\mathrm{IB}(\rho+\delta^{-1}\tau,\sigma)\right)=\frac{1}{\mathrm{Be}(\rho+\delta^{-1}\tau,\sigma)}-1,\]
and so we can see the above limit arises from the divergence of the first parameter of the inverse-beta distribution. Thus we recover the discussion of \cite[Section IV.D.2]{Thiery} (see also \cite{TD}).\\
(c) Due to the non-linearity of the change of variables and the appearance of the transposition map $\pi$ in Proposition \ref{p55}(c), one can not adopt the same approach to deduce that the partition function of the stationary beta random polymer converges under an associated scaling to that of the stationary gamma model. However, applying the symmetries of the distributions in question, it turns out that if $(\mu_\delta,\nu_\delta,\tilde{\mu}_\delta,\tilde{\nu}_\delta)$ are defined as in Proposition \ref{p55}(c), i.e.\ with parameters $(\sigma,\rho,\delta^{-1}\tau)$, and $(\mu'_\delta,\nu'_\delta,\tilde{\mu}'_\delta,\tilde{\nu}'_\delta)$ are defined similarly with parameters $(\rho,\sigma,\delta^{-1}\tau)$, then (surprisingly to us)
\[\left(\delta^{-1} I_d\times I_d\right)\left(\mu'_\delta\times\nu'_\delta\right)\sim\left(\delta^{-1} I_d\times (I\circ J\circ I)\right)\left(\mu_\delta\times\nu_\delta\right)\rightarrow \mu\times\nu,\]
\[\left(\delta^{-1} I_d\times \delta I_d\right)\left(\tilde{\mu}'_\delta\times\tilde{\nu}'_\delta\right)\sim\pi\circ\left(\delta I\times \delta^{-1} I\right)\left(\tilde{\mu}_\delta\times\tilde{\nu}_\delta\right)\rightarrow \tilde{\mu}\times \tilde{\nu},\]
where $(\mu,\nu,\tilde{\mu},\tilde{\nu})$ are defined as in Theorem \ref{t42}(b) and we use the symbol $\sim$ to simply mean that the distributions share the same limit as $\delta\rightarrow 0$. Moreover, considering the changes of variables in the left-hand sides above, one finds that
\[\lim_{\delta\rightarrow0}\left(\delta^{-1} I_d\times \delta I_d\right)\circ F_{(1,-1)}\circ\left(\delta I_d\times I_d\right)=F_{(1,0)},\]
which in turn implies
\begin{equation}\label{eqr}
\lim_{\delta\rightarrow0} \left( \delta^{-1} I_d\times I_d\right)\circ R_{(1,-1)}\circ\left(\delta I_d\times \delta I_d\times I_d\right)=R_{(1,0)}.
\end{equation}
In particular, it follows that
\[\left(\delta^{-n}Z^\delta_{n,m}\right)_{n,m \in \mathbb{Z}_+}\rightarrow \left(Z_{n,m}\right)_{n,m \in \mathbb{Z}_+}\]
in distribution as $\delta\rightarrow 0$, where $Z^\delta$ is the partition function of the stationary beta random polymer corresponding to $R_{(1,-1)}$, as defined from $\tilde{\mu}'_\delta$, $\mu'_\delta$ and $\nu'_\delta$, and $Z_{n,m}$ is the partition function of the limiting stationary gamma random polymer corresponding to $R_{(1,0)}$. The latter link between the beta and gamma models, as corresponding to the divergence of the second beta parameter, was previously observed in \cite{BC} (see \cite[Remark 1.5]{BC} in particular).\\
(d) The relation between the beta and inverse-gamma models that is stated as Proposition \ref{p55}(d) is apparently novel. Unlike in the previous case, however, we do not see a way to deduce from this the convergence of partition functions, if indeed this is possible.
\end{rem}

\subsection{Zero-temperature limits}\label{ztlsec}

Connecting positive- and zero-temperature bijections and detailed balance solutions, we have the following analogue of Propositions \ref{ptmap} and \ref{ptmeas}.

\begin{prop}\label{ztlmapmeas}(a) It holds that
\[ F_{\tBeta,\tBeta}^\z=F_{\AL,\AL},\qquad F_{\Gam,\tBeta}^\z=F_{\E,\AL},\qquad \left(F_{\Gam,\tBeta}^{-1}\right)^\z=F_{\E,\AL}^{-1},\]
with the convergence at \eqref{ztl} holding uniformly on compact subsets of $\mathbb{R}^2$.\\
(b)(i) For $\rho,\sigma,\tau>0$, it holds that
\begin{align*}
&\lim_{\varepsilon\rightarrow 0}S_\varepsilon^{-1}\left(\tBeta(\varepsilon \rho,\varepsilon \sigma)\times \tBeta(\varepsilon(\rho+\sigma),\varepsilon \tau)\times \tBeta(\varepsilon \tau,\varepsilon \sigma)\times\tBeta(\varepsilon(\sigma+\tau),\varepsilon \rho)\right)\\
&\qquad\qquad = \mathrm{AL}(\rho,\sigma)\times  \mathrm{AL}(\rho+\sigma,\tau)\times \mathrm{AL}(\tau,\sigma)\times \mathrm{AL}(\sigma+\tau,\rho).
\end{align*}
(ii) For $\rho,\sigma,\tau>0$, it holds that
\begin{align*}
&\lim_{\varepsilon\rightarrow 0}S_\varepsilon^{-1}\left(\Gam(\varepsilon \rho,e^{\tau/\varepsilon})\times \Gam(\varepsilon  \sigma,e^{\tau/\varepsilon})\times \Gam(\varepsilon (\rho+\sigma),e^{\tau/\varepsilon})\times \tBeta(\varepsilon  \rho,\varepsilon \sigma)\right)\\
&\qquad\qquad = \mathrm{sExp}(\rho,\tau)\times \mathrm{sExp}(\sigma,\tau)\times \mathrm{sExp}(\rho+\sigma,\tau)\times \AL(\rho,\sigma).
\end{align*}
\end{prop}
\begin{proof} The proof of part (a), which is straightforward, is omitted. Part (b) depends on two basic convergence results. The convergence of $S_\varepsilon^{-1}(\tBeta(\varepsilon \rho,\varepsilon \sigma))$ can be checked by a simple change of variables, similarly to the proof of Proposition \ref{ptmeas}. As for the convergence of $S_\varepsilon^{-1}(\Gam(\varepsilon \rho,e^{\tau/\varepsilon}))$, see \cite[Proposition 5.1(b)]{CS2}.
\end{proof}

As per the comments around \eqref{l1}-\eqref{m2}, together with Proposition \ref{gb}, the above result can be used to deduce that the limiting measures in parts (b)(i) and (b)(ii) are solutions of the detailed balance equation for $F_{\AL,\AL}$ and $F_{\E,\AL}$, respectively (cf.\ Proposition \ref{gbz}). In the same way, one can also make a statement about the detailed balance solutions of $F_{\E,\AL}^{-1}$. Note that this procedure misses the discrete solutions of the zero-temperature detailed balance equations.

\begin{rem}\label{ptztrem}
Either directly or by applying Corollary \ref{c32}, Proposition \ref{p34} and Proposition \ref{ztlmapmeas} (and equation \eqref{tildereq}), it is possible to obtain the continuous zero-temperature random polymer stationary measures of Theorem \ref{t45} from the positive-temperature ones of Theorem \ref{t42}. Although we omit details here, we note that the truncated asymmetric Laplace distribution emerges in the following way:
\begin{eqnarray*}
\lim_{\varepsilon\rightarrow 0}S_\varepsilon^{-1}\left(\mathrm{Beta}(\varepsilon \rho,\varepsilon \sigma)\right)&=&
\lim_{\varepsilon\rightarrow 0}S_\varepsilon^{-1}\circ \left(Q^{-1}\right)\left(\tBeta(\varepsilon \sigma,\varepsilon \rho)\right)\\
&=&\left(Q^{-1}\right)^\z\left(\lim_{\varepsilon\rightarrow 0}S_\varepsilon^{-1}\left(\tBeta(\varepsilon \sigma,\varepsilon \rho)\right)\right)\\
&=&\left(Q^{-1}\right)^\z \left(\AL(\sigma,\rho)\right)\\
&=&\AL(\rho,\sigma)\vee 0.
\end{eqnarray*}
Similarly to Remark \ref{ptptrem}, the connections between maps and stationary measures also allow us to relate the partition functions of the positive-temperature models to those of the zero-temperature models, and thereby recover known zero-temperature limits for each of the inverse-gamma, gamma, inverse-beta and beta random polymers. For further background, one can find discussion of zero-temperature limits relating to the inverse-beta, the gamma, and inverse-gamma polymers in \cite{Thiery}, and for the beta polymer in \cite[Section 4]{BC}.
\end{rem}

\subsection{Relationships between zero-temperature models}\label{ztsec}

The results corresponding to Propositions \ref{ptmap} and \ref{ptmeas} in the zero-temperature case are the following. Since the proofs of both are straightforward, they are omitted.

\begin{prop}\label{ztmap} On $(0,\infty)\times\mathbb{R}$, it holds that
\[\pi\circ\left(I^\z\times  I^\z_d\right)\circ F_{\AL,\AL}\circ \left(I^\z\times I^\z\right)=F_{\E,\AL},\]
and also
\[\left(I^\z\times I^\z\right)\circ F_{\AL,\AL}\circ\left(I^\z\times  I^\z_d\right)\circ\pi=F_{\E,\AL}^{-1}.\]
\end{prop}

\begin{prop}\label{ztmeas}
For $\rho,\sigma\in(0,\infty)$, it holds that
\[\lim_{\tau\rightarrow\infty}\left(I^\z\times I^\z\right)\left(\mathrm{AL}(\tau,\rho)\times\mathrm{AL}(\tau+\rho,\sigma)\right)=\mathrm{Exp}(\rho)\times \mathrm{Exp}(\sigma),\]
\[\lim_{\tau\rightarrow\infty}\pi\circ\left(I^\z\times  I^\z_d\right)\left(\mathrm{AL}(\sigma,\rho)\times \mathrm{AL}(\rho+\sigma,\tau)\right)=\mathrm{Exp}(\rho+\sigma)\times \AL(\rho,\sigma).\]
Also, with $(q,r)=(P,Q)$ for some $(P,Q)\in(0,1)$, it holds that
\[\lim_{p\rightarrow1}\left(I^\z\times I^\z\right)\left( \mathrm{sdAL}(p,q,m)\times \mathrm{sdAL}(pq,r,m)\right)= \mathrm{sGeo}(P,m)\times\mathrm{sGeo}(Q,m),\]
\[\lim_{p\rightarrow 1}\pi\circ\left(I^\z\times  I^\z_d\right)\left(\mathrm{sdAL}(r,q,m)\times\mathrm{sdAL}(qr,p,m)\right)=\mathrm{sGeo}(PQ,m)\times \mathrm{sdAL}(P,Q,m).\]
\end{prop}

Again, similarly to the comments around \eqref{l1}-\eqref{m2}, together with Proposition \ref{gbz}(b), Propositions \ref{ztmap} and \ref{ztmeas} imply that the limiting measures are solutions of the detailed balance equation for $F_{\E,\AL}$. Moreover, one can make a corresponding statement about the detailed balance solutions of $F_{\E,\AL}^{-1}$. Note that this procedure misses the shifts permitted in the distributions of Proposition \ref{gbz}(a). Hence, unlike the other links shown in Figure \ref{fig2}, which all preserve the three distributional parameters of the detailed balance solutions in question, the limits from $F_{\AL,\AL}$ to $F_{\E,\AL}$ or its inverse $F_{\E,\AL}^{-1}$ see a reduction from three to two parameters (cf.\ comment below Figure \ref{fig1}, and also \cite[Sections V.B.1,2]{Thiery}).

\subsubsection{Relations between zero-temperature random polymer models}

We now consider how the discussion of Subsection \ref{ptrem} applies in the zero-temperature case. Firstly, in view of Corollary \ref{c54}, one conjectures part (a) and (b) of the following statement. Moreover, from \eqref{tildereq} and \eqref{eqr}, one reasonably anticipates part (c). Since a direct proof of the claims is straightforward, we omit it.

\begin{prop}\label{ppp1}
(a) On $(-\infty,0]\times \mathbb{R}\times (-\infty,0]$, it holds that $R^\z_{(1,1)}=R^\z_{(0,1)}$.\\
(b) On $[0,\infty)\times \mathbb{R}\times (-\infty,0]$, it holds that $R^\z_{(1,1)}=R^\z_{(1,0)}$.\\
(c) On $(-\infty,0]\times [0,\infty)\times (-\infty,0]$, it holds that $\tilde{R}^\z_{(1,-1)}=R^\z_{(1,0)}\circ((Q^{-1})^\z\times I_d^\z\times I_d^\z)$.
\end{prop}

Towards connecting the stationary measures of the maps in question, we have the following statements. (As in the previous subsection, the limiting procedure does not yield the shifted exponential stationary solutions as limits.) Similar statements are possible for the discrete measures, but are omitted for brevity.

\begin{prop}\label{ppp2}
(a) For $\sigma,\tau > 0$, as $\rho\rightarrow \infty$:
\[\left(\mathrm{AL}(\rho,\sigma+\tau),\mathrm{AL}(\rho+\sigma,\tau),\mathrm{AL}(\rho,\sigma)\wedge 0\right)\rightarrow
\left(-\mathrm{Exp}(\sigma+\tau),-\mathrm{Exp}(\tau),-\mathrm{Exp}(\sigma)\right),\]
where we note that the limit measure is supported on $(-\infty,0]^3$.\\
(b) For $\rho,\sigma > 0$, as $\tau\rightarrow \infty$:
\[\left(\mathrm{AL}(\rho,\sigma+\tau),\mathrm{AL}(\rho+\sigma,\tau),\mathrm{AL}(\rho,\sigma)\wedge 0\right)\rightarrow
\left(\mathrm{Exp}(\rho),\mathrm{Exp}(\rho+\sigma),\AL(\rho,\sigma)\wedge 0\right),\]
where we note that the limit measure is supported on $[0,\infty)^2\times (-\infty,0]$.\\
(c) For $\rho,\sigma > 0$, as $\tau\rightarrow \infty$:
\begin{align*}
&\left((Q^{-1})^\z(\mathrm{AL}(\tau,\sigma)),\mathrm{AL}(\rho+\sigma,\tau)\vee 0, \mathrm{AL}(\sigma,\rho)\wedge0\right)\\
&\qquad\qquad\qquad\qquad\qquad\qquad\rightarrow
\left(\mathrm{Exp}(\sigma),\mathrm{Exp}(\rho+\sigma),\AL(\sigma,\rho)\wedge 0\right),
\end{align*}
where we note that the limit measure is supported on $[0,\infty)^2\times (-\infty,0]$.
\end{prop}

\begin{rem}\label{ztrem} (a) As in Remark \ref{ptptrem}, the links between maps and stationary measures of Propositions \ref{ppp1} and \ref{ppp2} allow us to connect the partition functions of the models in question. We note that the divergence of parameters corresponds with the positive-temperature results discussed in Remark \ref{ptptrem}(a)-(c). (In case (c), is it useful to note that on $(-\infty,0)$, the map $(Q^{-1})^\z$ is equal to the linear map $I^\z=-I_d^\z$.) Results of a similar nature connecting the Bernoulli-exponential polymer (that is, the zero-temperature limit of the inverse-beta polymer) or its discrete counterpart, the Bernoulli-geometric polymer, to exponential/geometric FPP/LPP have previously been described in \cite{Thiery}. We believe the conclusion for the Bernoulli-exponential/geometric FPP (that is, the zero-temperature limit of the beta polymer) that links this to the exponential/geometric FPP model, is new. Thus we choose to describe this explicitly: let $(Z^\tau_{n,m})_{n,m \in \mathbb{Z}_+}$ be the partition function of the zero-temperature random polymer corresponding to $\tilde{R}_{(1,-1)}$ with $(X_{n,m})_{n,m \in \mathbb{N}}$ being i.i.d.\ $\mathrm{AL}(\tau,\sigma)$ random variables, and boundary condition being given by independent random variables $(U_{n,0})_{ n \in \N}$ and $(V_{0,n})_{n \in \N}$ with marginal distributions $\mathrm{AL}(\rho+\tau,\sigma)$ and $\mathrm{AL}(\sigma,\rho)\wedge 0$, respectively, and $Z^\tau_{0,0}=0$, then
\[\left(Z^\tau_{n,m}\right)_{n,m \in \mathbb{Z}_+}\rightarrow \left(Z_{n,m}\right)_{n,m \in \mathbb{Z}_+}\]
in distribution as $\tau\rightarrow \infty$, where $(Z_{n,m})_{n,m \in \mathbb{Z}_+}$ is the partition function of the zero-temperature model corresponding to $R_{(0,1)}$, as defined from the limiting measures $\tilde{\mu}=\mathrm{Exp}(\sigma)$, $\mu=\mathrm{Exp}(\rho+\sigma)$ and $\nu=\mathrm{AL}(\sigma,\rho)\wedge 0$. As in the positive-temperature case, it does not seem possible to relate the zero-temperature limits of the beta and inverse-gamma models (i.e.\ Bernoulli-exponential/geometric FPP and exponential/geometric LPP, respectively) in a similarly convenient way.\\
(b) For each of the four zero-temperature models (i.e.\ $R^\z_{(0,1)}$, $R^\z_{(1,0)}$, $R^\z_{(1,1)}$, $\tilde{R}^\z_{(1,-1)}$), it is possible to show that the continuous stationary measures arise as limits of discrete versions of the stationary measures, and it follows that the corresponding partition functions also converge (cf. \cite{Thiery}).
\end{rem}

\section{Open problems/conjectures}\label{sec6}

We finish by listing a number of open problems and conjectures that arise from our study.

\begin{enumerate}
\item As we noted in the introduction, the discrete and ultra-discrete Toda lattice equations are related to $R_{(0,1)}$ and $R^\z_{(0,1)}$, respectively. Thus they are also connected to $F_{\Gam,\tBeta}$ and $F_{\E,\AL}$. Are there any natural discrete integrable systems linked to $F_{\tBeta,\tBeta}$ and $F_{\AL,\AL}$?
\item The characterisations of solutions to the detailed balance condition for the basic bijections $F_{\Gam,\tBeta}$, $F_{\tBeta,\tBeta}$ and $F_{\E,\AL}$ are complete (recall Propositions \ref{gb} and \ref{gbz}(a)). We conjecture that all the non-trivial solutions to the detailed balance condition for $F_{\AL,\AL}$ are given in the statement of Proposition \ref{gbz}(b). We note that, amongst absolutely continuous distributions satisfying certain technical conditions, this has recently been shown to be the case in \cite{BaoN}.
\item The characterisation of stationary solutions for the four positive-temperature random polymer models considered in this article is complete (recall Theorem \ref{t42}). In the zero-temperature case, it is only for $R_{(0,1)}^{\z}$ that we are able to make the same claim. Apart from the issue raised in the previous question, we also have the problem that the changes of scale in Proposition \ref{p34} are not bijective (for maps other than $R_{(0,1)}^{\z}$). Consequently, even establishing the conjecture concerning Proposition \ref{gbz}(b) will not be enough to complete the story for $R_{(1,0)}^{\z}$, $R_{(1,1)}^{\z}$ and $\tilde{R}_{(1,-1)}^{\z}$. Whilst it is tempting to conjecture that for these models, there are no other non-trivial stationary solutions than those given in Theorem \ref{t45}(b),(c),(d), this is apparently not the case. Indeed, \cite[Lemma 2.4]{CG18} yields another stationary measure for $R_{(1,0)}^{\z}$ given by, for some parameters $0<p<u\leq 1$, \[\tilde{\mu}=-\mathrm{Ber}(p),\qquad\mu=-\mathrm{Ber}(u),\qquad\nu=-\mathrm{Geo}\left(\frac{u-p}{u(1-p)}\right),\]
    where $\mathrm{Ber}$ represents a Bernoulli distribution and $\mathrm{Geo}$ a geometric one, that does not appear to correspond to a detailed balance solution for $F_{\E,\AL}$.
\item Related to the previous question, there are polymer models in the literature that do not fit into our framework, such as the discrete Hammersley process studied in \cite{CG19}. If one is able to develop general techniques for characterising the stationary solutions of maps of the form $R_{(\alpha,\beta)}^{\z}$, then it would be of interest to explore whether these apply in other settings.
\item As discussed in Section \ref{sec5}, connections between partition functions of random polymer models have been established in all the ways shown by the arrows on Figure \ref{fig1}. A notable omission here is a link between the partition functions of the beta and inverse-gamma models, in both the positive- and zero-temperature cases. (This point was also discussed in \cite[Section III.B]{TD}.) Even though we do not see such ourselves, we wonder whether there are any implications in this direction of the non-linear relationship between the maps and stationary measures of the two models made precise in Propositions \ref{p53}(d) and \ref{p55}(d).
\item A stronger form of stationarity that is known to hold for the positive-temperature beta-gamma models is `Burke's property', see \cite[Theorem 3.3]{S} for an explicit example and the discussion after \cite[Definition 2.2]{CN} for a more general comment. For such models, Burke's property can be understood as a consequence of the equivalence between the stationary measures for $R_{(\a,\b)}=\overline{F}_{(\a,\b)}^{(2,3)}$ and those that are invariant under $\overline{F}_{(\a,\b)}$ (i.e.\ Proposition \ref{twothree}). The fact that one can not express  $R_{(\a,\b)}^{\z}$ in the same way in general means that further work is required to check whether the stationary measures we present in the zero-temperature case satisfy a version of Burke's property. Note we do know that $R_{(0,1)}^{\z}$ has a characterisation as $(\overline{F}_{(0,1)}^\z)^{(2,3)}$, and so in this case we can recover the Burke's property of \cite{BCS}. On the other hand, whilst we do not have such a description for $R_{(1,0)}^{\z}$, it is known that the stationary measure for $R_{(1,0)}^{\z}$ of \cite{CG18} nonetheless satisfies a version of Burke's property.
\item For random polymer models, there are different notions of solvability. As well the computability of stationary measures, these include Bethe ansatz integrability and presence of the Robinson-Schensted-Knuth correspondence. A summary of what is known for some of the models considered here appears in \cite[Figure 1]{Thiery}. One is naturally led to consider whether any of the connections between models described in this article offer insight into links between other solvability properties. Moreover, an important aspect of current research on random polymers concerns their place in the KPZ universality class, see \cite{Corwinsurvey} for an introduction to this area and \cite{BR} and the references therein for a description of some more recent developments regarding random polymers. It would of course be worthwhile to consider whether our results provide any further insight in this direction.
\end{enumerate}

\section*{Acknowledgements}

This research was supported by JSPS Grant-in-Aid for Scientific Research (B), 19H01792. The research of DC was also supported by JSPS Grant-in-Aid for Scientific Research (C), 19K03540, and the Research Institute for Mathematical Sciences, an International Joint Usage/Research Center located in Kyoto University. The authors thank Kevin Bao and Christian Noack for sharing a preliminary version of their results from \cite{BaoN}, and Ivan Corwin and Mark Rychnovsky for helpful comments. This work was completed while MS was kindly being hosted by the Courant Institute, New York University. MS is further grateful to Ivan Corwin for arranging space to work in Columbia University during this extended visit to New York. Finally, we thank the referees who commented on earlier versions of this work, as their insightful comments helped improve the article and contributed to several of the open problems.

\appendix

\section{Probability distributions}\label{appa}

In the following list, we give definitions of the various probability distributions that appear within this article.

\begin{description}
\item[Shifted exponential distribution] For $\lambda >0$, $c\in \R$,  the \emph{shifted exponential} distribution with parameters $(\lambda,c)$, which we denote $\mathrm{sExp}(\lambda,c)$, has density
\[\frac{1}{Z}e^{-\la x}\mathbf{1}_{[c,\infty)}(x),\qquad x\in\R,\]
where $Z$ is a normalizing constant. We further write $\mathrm{Exp}(\lambda):=\mathrm{sExp}(\lambda,0)$ for the standard \emph{exponential} distribution.
\item[Shifted scaled geometric distribution] For $\theta\in(0,1)$, $M\in\mathbb{Z}$ and $m\in(0,\infty)$, we say a random variable $X$ has \emph{shifted scaled geometric distribution} with parameters $1-\theta$, $M$ and $m$ if
\[\mathbf{P}\left(X=mx\right)=\frac{1}{Z}\theta^{x},\qquad x\in\{M,M+1,\dots\},\]
where $Z$ is a normalising constant; in this case we write $X\sim \mathrm{ssGeo}(1-\theta,M,m)$. We further write $\mathrm{sGeo}(1-\theta,m):=\mathrm{ssGeo}(1-\theta,0,m)$ for the \emph{scaled geometric distribution} with parameters $1-\theta$ and $m$.
\item[Asymmetric Laplace distribution] For $\lambda_1,\lambda_2\in(0,\infty)$, the \emph{asymmetric Laplace} distribution with parameters $(\lambda_1,\lambda_2)$, which we denote $\mathrm{AL}(\lambda_1,\lambda_2)$, has density
\[\frac{1}{Z}\left(e^{-\lambda_1 x}\mathbf{1}_{(0,\infty)}(x)+e^{\lambda_2 x}\mathbf{1}_{(-\infty,0)}(x)\right),\qquad x\in\R,\]
where $Z$ is a normalizing constant.
\item[Scaled discrete asymmetric Laplace distribution] For $\theta_1,\theta_2\in(0,1)$ and $m \in (0,\infty)$, we say a random variable $X$ has \emph{scaled discrete asymmetric Laplace} distribution with parameters $(1-\theta_1,1-\theta_2,m)$ if
\[\mathbf{P}\left(X=mx\right) =\left\{\begin{array}{ll}
                    \frac{1}{Z}\theta_1^{x}, & x \in \{0,1,2,\dots\}, \\
                    \frac{1}{Z}\theta_2^{-x}, & x \in \{ \dots, -2,-1\},
                  \end{array}\right.\]
where $Z$ is a normalizing constant; in this case we write $X \sim \mathrm{sdAL}(1-\theta_1,1-\theta_2,m)$.
\item[Gamma distribution] For $\lambda,c\in(0,\infty)$, the \emph{gamma} distribution with parameters $(\lambda,c)$, which we denote $\mathrm{Gam}(\lambda,c)$, has density
\[\frac{1}{Z}x^{\lambda-1}e^{-cx}\mathbf{1}_{(0,\infty)}(x),\qquad x\in\mathbb{R},\]
where $Z$ is a normalizing constant.
\item[Inverse-gamma distribution] For $\lambda,c\in(0,\infty)$, the \emph{inverse-gamma} distribution with parameters $(\lambda,c)$, which we denote $\mathrm{IG}(\lambda,c)$, has density
\[\frac{1}{Z}x^{-\lambda-1}e^{-cx^{-1}}\mathbf{1}_{(0,\infty)}(x),\qquad x\in\mathbb{R},\]
where $Z$ is a normalizing constant.
\item[Beta distribution] For $\lambda_1,\lambda_2 \in(0,\infty)$, the \emph{beta} distribution with parameters $(\lambda_1,\lambda_2)$, which we denote $\mathrm{Be}(\lambda_1,\lambda_2)$, has density
\[\frac{1}{Z} x^{\lambda_1-1}(1-x)^{\lambda_2-1}\mathbf{1}_{(0,1)}(x),\qquad x\in\mathbb{R},\]
where $Z$ is a normalizing constant.
\item[Inverse-beta distribution] For $\lambda_1,\lambda_2 \in(0,\infty)$, the \emph{inverse-beta} distribution with parameters $(\lambda_1,\lambda_2)$, which we denote $\mathrm{IB}(\lambda_1,\lambda_2)$, has density
\[\frac{1}{Z} (x-1)^{\lambda_2-1}x^{-\lambda_1-\lambda_2}\mathbf{1}_{(1,\infty)}(x),\qquad x\in\mathbb{R},\]
where $Z$ is a normalizing constant.
\item[Beta prime distribution] For $\lambda_1,\lambda_2 \in(0,\infty)$, the \emph{beta prime} distribution with parameters\linebreak
    $(\lambda_1, \lambda_2)$, which we denote $\tBeta(\lambda_1,\lambda_2)$, has density
\[\frac{1}{Z} x^{\lambda_1-1}(1+x)^{-\lambda_1-\lambda_2}\mathbf{1}_{(0,\infty)}(x),\qquad x\in\mathbb{R},\]
where $Z$ is a normalizing constant.
\end{description}

\bibliography{irf}
\bibliographystyle{amsplain}

\end{document}